\documentclass[12pt, a4paper, reqno, twoside, makeidx]{amsart}
\usepackage{graphicx}
\usepackage[margin=2.8cm, marginpar=2cm]{geometry}
\usepackage[latin1]{inputenc}
\usepackage{amsmath, amsthm, amssymb, amsfonts, amscd}
\usepackage[english]{babel}
\usepackage{mathdots}
\usepackage[all]{xy}
\usepackage{hyperref}
\usepackage{url}
\usepackage{color}
\usepackage{faktor}
\usepackage{pinlabel}

\newtheorem{thm}{Theorem}[section]

\newtheorem{prop}[thm]{Proposition}

\newtheorem{defi}[thm]{Definition}
\theoremstyle{remark}
\newtheorem{ex}[thm]{Example}
\newtheorem{rmk}[thm]{Remark}
\newtheorem*{theorem*}{Theorem}
\newtheorem*{cor*}{Corollary}
\newtheorem*{proposition*}{Proposition}

\newcommand{\Z}{\mathbb{Z}}
\newcommand{\Q}{\mathbb{Q}}
\newcommand{\R}{\mathbb{R}}

\newcommand{\F}{\mathbb{F}}

\newcommand{\OO}{\mathbb{O}}
\newcommand{\XX}{\mathbb{X}}

\newcommand{\spinc}{Spin^c}
\newcommand{\zetap}{\Z_p}
\newcommand{\lp}{L(p,q)}
\newcommand{\permut}{\mathfrak{S}_n}
\newcommand{\tildec}{\widetilde{\mathrm{GC}}}
\newcommand{\tildeh}{\widetilde{\mathrm{GH}}}
\newcommand{\hatc}{\widehat{\mathrm{GC}}}
\newcommand{\hath}{\widehat{\mathrm{GH}}}
\newcommand{\minusc}{\mathrm{GC}^-}
\newcommand{\minush}{\mathrm{GH}^-}
\newcommand{\taugen}[2]{\widetilde{\tau}_{#1,#2}}
\newcommand{\funct}[3]{#1 : #2 \longrightarrow #3}
\newcommand{\iab}[2]{\mathcal{I}(#1, #2)}
\newcommand{\segno}[1]{\mathcal{S}(#1)}
\newcommand{\generalperm}{\widetilde{\mathfrak{S}}_n}

\title[A note on grid homology in lens spaces]{A note on grid homology in lens spaces:\\$\Z$ coefficients and computations}
\author{Daniele Celoria}
\date{}
\begin{document}
\maketitle
\begin{abstract}
We present a combinatorial proof for the existence of the sign refined grid homology in lens spaces, and a self contained proof that $\partial_\Z^2 = 0$. We also present a Sage program that computes $\widehat{\mathrm{GH}} (L(p,q),K;\mathbb{Z})$, and provide empirical evidence supporting the absence of torsion in these groups.
\end{abstract}

\section{Introduction}\label{sec:intro}
Ozsv\'ath and Szab\'o's Heegaard Floer homology \cite{holdisk3man} is undoubtedly one of the most powerful tools of recent discovery in low dimensional topology. It has far-reaching consequences and has been used to solve long-standing conjectures (for a survey of some results see \emph{e.g.}~\cite{ozsvath2006introduction}, \cite{andrassurvey} and \cite{homsurvey}). 

Roughly speaking, it associates\footnote{There are actually many different variants of the theory, which we ignore presently; in the following sections we will define some variants which will be relevant to our discussion.} a graded group to a closed and oriented 3-manifold $Y$, the \emph{Heegaard Floer homology} of $Y$, by applying a variant of Lagrangian Floer homology in a high dimensional manifold determined by a Heegaard decomposition of $Y$. 

Soon after its definition, it was realised independently in \cite{holdisknot} and \cite{rasmussenknot} that a nullhomologous knot $K \subset Y$ induces a  filtration on the  complex whose homology is the Heegaard Floer homology of $Y$. Furthermore, the filtered quasi-isomorphism type of this complex is an invariant of the couple $(Y,K)$, denoted by $\mathrm{HFK}(Y,K)$.

The major computational drawback of these theories lies in the differential, which is defined through a count of pseudo-holomorphic disks with appropriate boundary conditions. 
Nonetheless, a result of Sarkar and Wang \cite{sarkarwang} ensures that --after a choice of a suitable doubly pointed Heegaard  diagram $\mathcal{H}$ for $(Y,K)$-- the differential can be computed directly from the combinatorics of $\mathcal{H}$. If, moreover, $Y$ is a rational homology 3-sphere admitting a Heegaard splitting of genus $1$ (\emph{i.e.} $Y = S^3$ or $Y$ is a lens space $\lp$), the homology $\mathrm{HFK}(Y,K)$ admits a neat combinatorial definition, known as \emph{grid homology}.

Grid homology in  $S^3$ was pioneered by Manolescu, Oszv\'{a}th and Sarkar in \cite{manolescu2009combinatorial}, and for lens spaces by Baker, Hedden and Grigsby in \cite{BGH}; as the name suggests, both the ambient manifold and the knot are encoded in a grid, from which complex and differential for the grid homology can be extracted by simple combinatorial methods.\\

After establishing the necessary background, in Section~\ref{sec:gridhom} we will present the definition of grid homology in lens spaces as given in \cite{BGH}; here, we produce a purely combinatorial proof that $\partial^2 = 0$. Note that the analogous proof in~\cite{BGH} is indirect, and relies on the well-definedness of the analytic theory.

The grid homology in lens spaces from~\cite{BGH} is only defined over $\F_2$. We provide a lift to integer coefficients by proving the existence and uniqueness for \emph{sign assignments} of grid diagrams. A sign assignment is a coherent choice of sign for each term appearing in the differential of the grid chain complex.

\begin{thm}\label{thm:main}
Sign assignments exist on all grids representing a knot $K\subset \lp$, and can be described combinatorially. Moreover, for a fixed grid diagram, the sign refined grid homology does not depend on the choice of a sign assignment.
\end{thm}
The proof of this theorem is the content of Section \ref{sec:zcoeff}. 
The sign refinement of the theory is carried on using a group theoretic reformulation of sign assignments due to Gallais \cite{gallais2008}.

Note that the theory developed in this paper is not sufficient to establish the fact that the sign refined grid homology is an invariant of links in lens spaces (the analogous result in~\cite{BGH} follows from an isomorphism with the analytical theory).
This flaw has however been addressed more recently in the paper~\cite{tripp} by Tripp, where invariance of grid homology with integer coefficients is proved.\\

Finally, in Section \ref{sec:computations} we present some computations and examples, together with a description of the program used to make them.
An interactive online version of this program is freely available on my homepage \cite{homepage}. 
With this tool, we are able to show that small grids (see the discussion in Section~\ref{sec:computations}) have torsion free grid homologies:
\begin{prop}
The sign refined grid homology of knots with small parameters is torsion free.
\end{prop}
This result provides empirical evidence for the absence of torsion in the knot Floer homology of knots in lens spaces. Analogous results for knots in the three sphere have been found by Droz in \cite{droz2008effective}.

\subsection*{Acknowledgments}
I would  like to thank my advisor Paolo Lisca and Andr{\'a}s Stipsicz for suggesting this topic, Paolo Aceto, Marco Golla, Enrico Manfredi, Francesco Lin and Agnese Barbensi for useful and interesting conversations and support. I would also like to thank the University of Pisa for the hospitality and the computational resources provided. Finally, I want to thank the referee for their detailed comments that helped to substantially improve the manuscript.

\section{Grid homology in lens spaces}\label{sec:gridhom}
\subsection{Representing knots with grids}\label{ssect:s3}
In what follows $p,q$ will always be two coprime integers, and the lens space $\lp$ is the (closed, oriented) 3-manifold obtained by $-\frac{p}{q}$ surgery on the unknot $\bigcirc \subset S^3$; to avoid confusion, a knot $K$ in a 3-manifold $Y$ will usually be denoted by $(Y,K)$.
\begin{defi}
Consider  a $n \times pn$ grid in $\R^2$, consisting of the segments 
 $\widetilde{\alpha}_i = (tnp, i)$ and $\widetilde{\beta}_j = (j, tn)$ with $i\in \{ 0, \ldots, n\} , \:j\in\{0, \ldots, np\}$ and  $t\in [0,1]$.
A \emph{twisted grid  diagram for $\lp$} is the grid on the torus given  by identifying $\widetilde{\beta}_0$ to $\widetilde{\beta}_{pn}$, and then $\widetilde{\alpha}_0$ to $\widetilde{\alpha}_{n}$ according to a twist depending on $q$ (see Figure~\ref{fig:identificazioni}): $$\alpha_{n} \ni (s,n) \sim (s-qn \:  (\text{mod }\;pn), 0) \in \alpha_{0}$$
Here $s \in [0,pn]$; the condition $(p,q) = 1$ guarantees that after the  identifications the planar grid becomes a toroidal grid.\\
Call $\alpha = \{ \alpha_i\} $ and $\beta = \{ \beta_i \}$ with $i \in \{1,\ldots ,n\}$ the $n$ horizontal (respectively vertical) circles obtained after the identifications in the grid.
\end{defi}

We can encode a link $L$ in $L(p,q)$ by placing a suitable version of the $\mathbb{X}$'s and $\mathbb{O}$'s for grid diagrams in $S^3$:
let $\mathbb{X} = \{ \mathbb{X}_i\}$ and $\mathbb{O} = \{ \mathbb{O}_i\},$ with $i=1, \ldots, n$ be two sets of markings.
Put each one of them in the little squares\footnote{For concreteness, think of the markings as having half integral coordinates in the planar grid.} of $G \setminus  \left(\alpha \cup  \beta \right)$ in such a way that each column\footnote{Beware! In a twisted toroidal grid a column ``wraps around'' a row $p$ times.} and row contains exactly one element of $\,\mathbb{X}$ and one of $\,\mathbb{O}$, and each square contains at most one marking. By rows and columns here we mean the regions of a grid --homeomorphic to annuli-- bounded by two consecutive $\alpha$ or $\beta$ curves respectively.

Now join with a segment each $\mathbb{X}$ to the $\mathbb{O}$ which lies on the same row, and each $\mathbb{O}$ to the $\mathbb{X} $ which lies on the same column (keeping in mind the twisted identification); with this convention we can encode an orientation\footnote{Note that this convention is the opposite of the one used in \cite{SOS}, but agrees with the one of \cite{BGH}; see also Remark \ref{rmk:opposite}.} for the link. To get an honest link, just remove self-intersections by converting 
each self-intersection to an overcrossing of the vertical segments over the horizontal ones (as in Figure~\ref{fig:griglia}).

The grid together with the markings is known as a multipointed Heegaard diagram for $(\lp, K)$ and $$|\alpha_i \cap \beta_j| = p \;\; \forall i,j \in \{1,\ldots ,n\}.$$
\begin{rmk}
There are two possible ways to connect each $\mathbb{X}_i$ to the corresponding $\mathbb{O}$ marking on the same row\slash column, but the isotopy class of the resulting link does not depend upon the possible choices. Indeed, the two links given by these different choices for each row\slash column are isotopic. This follows at once from the fact that the two possible arcs connecting two markings on the same row\slash column are --by construction-- related by a slide along a meridional disk of the Heegaard decomposition of $\lp$. In other words, the two arcs form a decomposition of the boundary of a disk that defines the given Heegaard splitting.
\end{rmk}

The integers $n,p$ and $q$ will be called the \emph{parameters} of the grid diagram $G$; the $p$ squares of height/length $n$ obtained by cutting the torus along $\alpha_1$  and $\beta_1$  (in the planar representation of the grid) are  called \emph{boxes}. 
It is worth to point out that the case in which $p=1$ and $q=0$  gives as expected a usual grid diagram for a link in $S^3$.
We will often deliberately forget the distinction between planar and toroidal grids, according to the motto ``draw on a plane, think on a torus''.

\begin{figure}[ht]
\includegraphics[width=10cm]{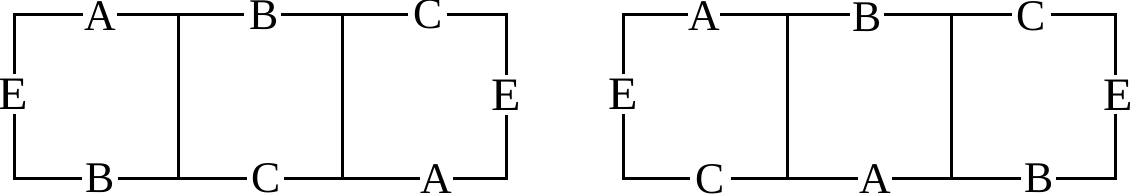}
\caption{Top-bottom identifications for a 3 dimensional grid for $L(3,1)$ on the left and $L(3,2)$ on the right.}
\label{fig:identificazioni}
\end{figure}
\begin{figure}[ht]
\includegraphics[width=9cm]{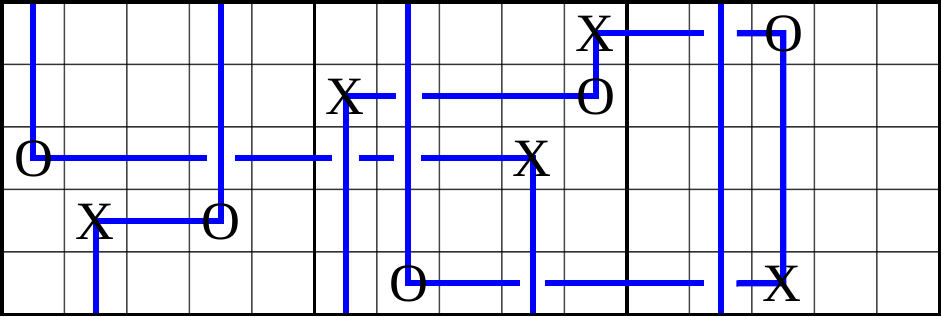}
\caption{The link obtained by joining $\mathbb{X}$'s and $\mathbb{O}$'s in a grid for $L(3,2)$ of grid dimension 5. Boxes are delimited by thicker black lines.}
\label{fig:griglia}
\end{figure}
\begin{rmk}\label{rmk:opposite}
Exchanging the role of the markings in a grid representing a knot $K$ produces a grid diagram for the same link with the opposite orientation on each component.
\end{rmk}

The following result is a consequence of \cite[Theorem~1.1]{BGH} and \cite[Section~2]{cromwell1995embedding}.
\begin{prop}
Every link in $L(p,q)$ can be represented by a grid diagram; two different grid representations of a link differ by a finite number of grid moves analogous to  Cromwell moves for grid diagrams in $S^3$:
\begin{itemize}
\item \emph{Translations:} these are just vertical and horizontal integer shifts of the grid (respecting the twisted identifications). 
\item \emph{(non-interleaving) Commutations:} if two adjacent row/columns $c_1$ and $c_2$ are such that the  markings of $c_1$ are contained in a connected component of $c_2$ with the two squares containing the markings removed, then they can be exchanged.
\item \emph{(de)Stabilisation:} these are the only moves that  change the dimension of the grid. There are 8 types of stabilisations, as shown in Figure~\ref{fig:gridmoves}. Destabilisations are just the inverse moves.
\end{itemize}
\begin{figure}[ht]
\includegraphics[width=12cm]{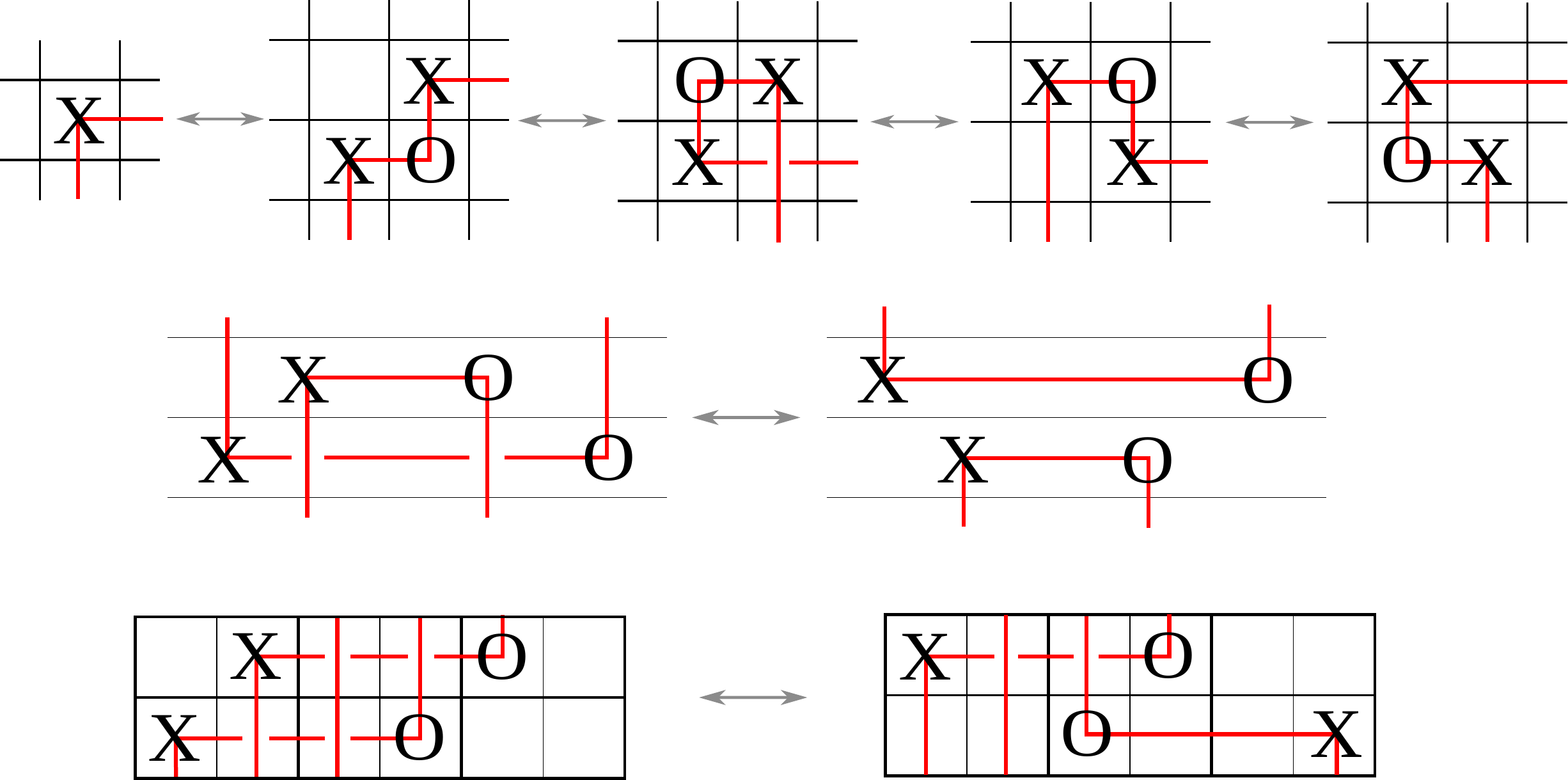}
\caption{Some examples of grid moves; in the top row of the figure we see four different kinds of stabilisations (there are other four where the roles of the markings are exchanged). In the middle, a schematic example of a row commutation. The bottom part of the figure displays an example of vertical translation in a grid of dimension $2$ for a knot in $L(3,1)$.}
\label{fig:gridmoves}
\end{figure}
\end{prop}
The homology class of a knot $K \subset \lp$ can be read directly from the grid; we just need to keep track of the signed number of intersections of the knot with a meridian of the torus. With the orientation conventions we have established (so that vertical arcs connect $\OO$'s to $\XX$'s):
$$H_1(\lp ;\Z) \ni [K] = \# \{ \alpha_1 \cap K\}   \;\;  (\text{mod }\;p).$$
If $G$ is a grid of parameters $(n,p,q)$, we call $n$ the \emph{dimension} or \emph{grid number} of $G$. The term \emph{minimal grid number} will be used when referring to the isotopy class of a knot $(\lp,K)$; in this case, we mean the quantity
$$GN(K) = \min \{n \:|\: G \mbox{  is a grid with parameters } (n,p,q) \mbox{ representing K}  \}.$$

\subsection{Generators of the complex}\label{ssec:complex}
In the following, we are going to define two different versions (sometimes also known as \emph{flavours}) of the grid homology for knots in lens spaces. They can both be defined by slight variations in the complex, the ground ring or the differential we will introduce below.
For clarity, we are going to restrict ourselves to $\F = \Z_2$ coefficients until the next section, and to knots throughout the paper.
\begin{defi}\label{def:generators}
Given a grid $G$ of dimension   $n$ representing a knot $K \subset \lp$, the \emph{generating set for $G$} is the set $S(G)$ comprising all bijections between $\alpha$ and $\beta$ curves. This corresponds to choosing $n$ points in $\alpha \cap \beta$ such that there is exactly one on each $\alpha $ and $ \beta$ curve.
There is a bijection $$S(G) \longleftrightarrow \permut \times \zetap^n$$ which can be described as follows:
since we fixed a cyclic labelling of the $\alpha $ and $ \beta$ curves it makes sense to speak of the $m$-\emph{th} intersection between two curves, with $0\le m\le p-1$; so if the $l$-\emph{th} component of a generator lies on the $m$-\emph{th} intersection of $\alpha_l$ and $\beta_j$ then the permutation $\sigma \in \permut$ associated will be such that $\sigma (l) = j$ and the $l$-\emph{th} component of $ \zetap^n$ will be $m$ (see Figure~\ref{fig:generatoritwisted}).\\
If $x \in S(G)$, we can thus write $x = (\sigma_x , (x_1^p, \ldots , x_n^p))$; we will refer to $\sigma_x$ as the \emph{permutation} component of the generator, and to $(x_1^p, \ldots , x_n^p)$ as its \emph{$p$-coordinates}.
\begin{figure}
\includegraphics[width =7cm]{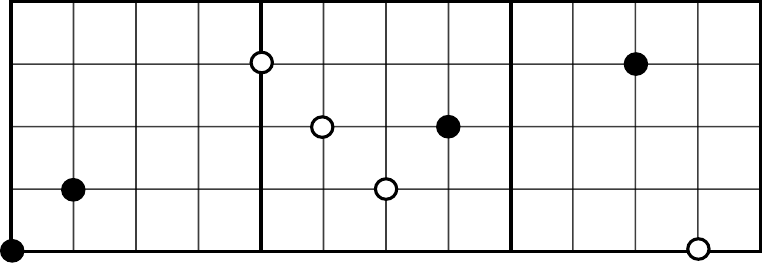}
\caption{Under the bijection described in Definition~\ref{def:generators} the white generator corresponds to $\left((14)(23),(2,1,1,1) \right)$, and the black to $\left((34),(0,0,1,2)\right) \in \mathfrak{S}_4 \times \Z_3^4$. }
\label{fig:generatoritwisted}
\end{figure}
\end{defi}
$S(G)$ can be endowed with a $\left(\Q,\Q, \zetap \right)$-valued  grading.  The first two degrees are known as \emph{Maslov} and  \emph{Alexander} degrees. The last one is the $\spinc$ degree; since it is preserved by the differential (Proposition \ref{teo:grosso}), it will provide a splitting of the complex into $p$ direct summands. All these degrees are going to be defined in a purely combinatorial way.
\begin{defi}
Let $A $ and $B$ denote two finite sets of points in  $\R^2$; let $\iab{A}{B}$ be the number of pairs $$((a_1,a_2),(b_1,b_2)) \subset A\times B$$ such that $a_i<b_i$ for $i=1,2$.\\
Denote by  $X(p,n)$ (respectively $Y(p,n) $) the set of $n$-tuples (respectively $pn $-tuples) of points contained in the $n \times pn$ 
(respectively $pn \times pn$) rectangle in $\R^2$ whose bottom vertices are $(0,0)$ and $(pn,0 )$; then define
$$C_{p,q} : X(p,n) \longrightarrow Y(p,n) $$
as the function sending a $n$-tuple $\{ (c_i , b_i) \}_{i=1, \ldots, n}$ to the $pn$-tuple
$$\{ (c_i + nqk\:\:  (\text{mod }\;np), b_i + nk) \}_{\substack{i=1, \ldots, n \\ k=0, \ldots, p-1}}$$
In order to avoid notational overloads, we are going to write $\widetilde{x} $ instead of $C_{p,q}(x)$; one example of such a function is described in Figure~\ref{fig:cpq}.
\begin{figure}[ht]
\includegraphics[width=11cm]{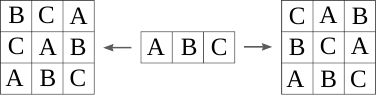}
\caption{A representation of the action of $C_{p,q}$ for $(p,q) =(3,1)$ and $(3,2)$ (on the left and right respectively).}
\label{fig:cpq}
 \end{figure}
\end{defi}
We can then define the Maslov degree as follows:
\begin{equation}\label{eqn:maslov}
M(x) = \frac{1}{p} \left[ \iab{\widetilde{x}}{\widetilde{x}} - \iab{\widetilde{x}}{\widetilde{\mathbb{O}}} -\iab{\widetilde{\mathbb{O}}}{\widetilde{x}}
  + \iab{\widetilde{\mathbb{O}}}{\widetilde{\mathbb{O}}}\right] + d(p,q,q-1) + 1
\end{equation}
In the equation above, $d(p,q,q-1)$ is a rational number known as the \emph{correction term} of $\lp$ associated to the $(q-1)$\emph{-th} $\spinc$ structure; as explained in \cite[Prop.~4.8]{absolutely}, it can be computed recursively as follows\footnote{A user-friendly online calculator for these correction terms can be found at \cite{homepage}.}:
\begin{itemize}
\item $d(1,0,0) = 0$
\item $d(p,q,i) = \left( \frac{pq - (2i +1 -p - q)^2}{4pq} \right) + d(q,r,j)$  where $r$ and $j$ denote the reduction of $p$ and $i$ $(\text{mod } \,q)$.
\end{itemize}
Similarly, the Alexander grading can be defined as:
\begin{equation}\label{eqn:alex}
A(x) = \frac{1}{2p} \left[ \iab{\widetilde{\mathbb{O}}}{\widetilde{\mathbb{O}}} -  \iab{\widetilde{\mathbb{X}}}{\widetilde{\mathbb{X}}}   + 2        \iab{\widetilde{\mathbb{X}}}{\widetilde{x}} - 2     \iab{\widetilde{\mathbb{O}}}{\widetilde{x}} \right]           +\frac{1-n}{2}
\end{equation}

By slightly modifying the differential we'll introduce in the next section, $A$ can be demoted to a filtration on the complex, rather than a degree. The complexes we are going to consider should be thought of as the graded objects associated to this filtration.

\begin{rmk}
Note that Equation~\eqref{eqn:alex} is not the standard formula used to define $A$; here we are using the fact (see \cite[Sec.~2]{degany2008some}) that in a grid of dimension $n$ for a knot in $S^3$ $$\iab{x}{J} - \iab{J}{x} =n,$$ with $J = \OO$ or $\XX$, and $x$ is any generator in $S(G)$.
\end{rmk}
Now call $(a_1^{\OO}, \ldots , a_n^{\OO})$ the $p$-coordinates of the generator whose components are in the lower left vertex of the squares  which contain a $\OO$ marking.
The $\spinc$ degree of $x = \left(\sigma_x , (a_1, \ldots, a_n)\right) \in S(G)$ is defined as:
$$S:  \mathfrak{S}_n \times \zetap^n \longrightarrow \zetap $$
\begin{equation}\label{eqn:spinc}
S(x) =  q-1 + \sum_{i=1}^{n} \left( a_i -  a_i^\mathbb{O} \right) \:\:\: (\text{mod }\;p)
\end{equation}
We are implicitly using a canonical identification between $\spinc (\lp)$ and $\zetap$ (\emph{cf.} \cite[Sec. 4.1]{absolutely}).

It is clear from the definition that the Alexander grading depends on the placement of all the markings, while $M$ and $S$ only on the position of the $\OO$s.\\

Let $R = \F [V_1, \ldots, V_n]$ denote the ring of  $n$-variable polynomials with $\F$ coefficients, and $\widehat{R} = \faktor{R}{\{ V_1 = 0\}}$.
These $V$ variables\footnote{We adopt here the convention of \cite{SOS}, in order to stress the difference between the endomorphisms of the complex (the $V_i$'s) and the induced map on homology, which will be denoted by $U$.} are graded endomorphisms of the complex; their function is to ``keep track'' of the $\OO$ markings in the differential.\\
We can now define at least the underlying module structure of the complexes we are going to use in what follows:
\begin{defi}
 The \emph{minus} complex $\mathrm{GC}^-(G)$ is the free $R$-module generated over $S(G)$.
The \emph{hat} complex $\widehat{\mathrm{GC}}(G)$ is the free $\widehat{R}$-module generated over $S(G)$.
Extend the gradings to the whole module by setting the behaviour of the action for the variables in the ground ring:
\begin{itemize}
\item[] $A(Vx) = A(x) - 1$
\item[] $M(V x) = M(x) - 2$
\item[] $S(V x) = S(x)$
\end{itemize}
where $V$ is any of the $V_i$.
\end{defi}
\begin{ex}\label{esempio}
In this example we are going to exhibit the generating set of the grid $G$ on the left of Figure~\ref{fig:esempio}, in the 0-th $\spinc$ structure, which we are going to denote by $S(G,0)$.
$S(G,0)$ consists of four elements (see also Figure~\ref{fig:cpxx}):$$a = \F_{\left[-\frac{1}{4},-\frac{1}{4}\right]} , b = \F_{\left[-\frac{1}{4},-\frac{1}{4}\right]}, c = \F_{\left[\frac{3}{4},-\frac{1}{4}\right]},d = \F_{\left[-\frac{5}{4},-\frac{5}{4}\right]}$$
\begin{figure}[ht]
\includegraphics[width=6cm]{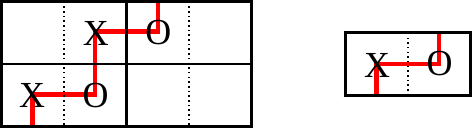}
\caption{A grid for the knot considered in Example~\ref{esempio}, and the grid obtained after a destabilization. Confront the computations in this example with Remark~\ref{rmk:semplici}.}
\label{fig:esempio}
\labellist
\pinlabel $d$ at 52 310
\pinlabel $a$ at 226 502
\pinlabel $b$ at 226 467
\pinlabel $c$ at 420 485
\pinlabel $A$ at 338 617
\pinlabel $M$ at 499 559
\pinlabel $V^2d$ at 700 153
\pinlabel $V^2c$ at 715 283
\pinlabel $V^2a$ at 580 315
\pinlabel $V^2b$ at 580 255
\pinlabel $Vd$ at 865 283
\pinlabel $Vb$ at 749 400
\pinlabel $Va$ at 749 448
\pinlabel $Vc$ at 875 424
\pinlabel $d$ at 1040 424
\pinlabel $b$ at 940 539
\pinlabel $a$ at 940 579
\pinlabel $c$ at 1040 559
\pinlabel $A$ at 1038 697
\pinlabel $V$ at 1100 633
\endlabellist
\includegraphics[width=12cm]{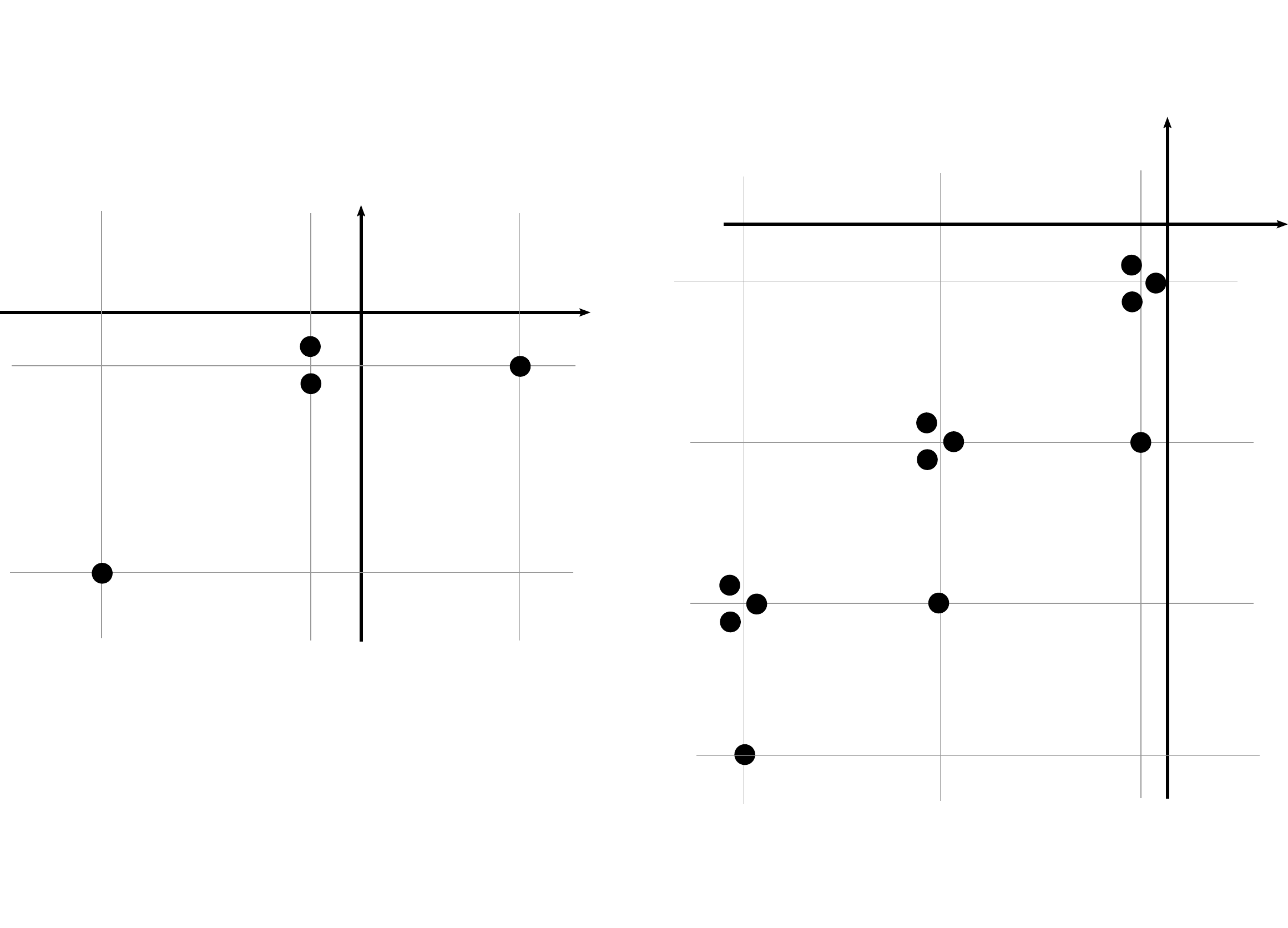}
\caption{The generating set $S(G,0)$, with the bi-degree $(M,A)$ on the axes is displayed on the left. On the right instead we have the underlying modules for the complexes $\hatc (G,0) = \minusc (G,0)$  with axes labelled by powers of the $V$ variables and Alexander degree. The dots represent generators over $\F$.}
\label{fig:cpxx}
\end{figure}
The notation $\F_{[a,b]}$ denotes a generator having $(M,A) = (a,b)$ bi-degree.
\end{ex}

\subsection{The differential}\label{ssec:diff}
As already mentioned in the introduction, grid homology hinges upon Sarkar and Wang's main result in~\cite{sarkarwang}; in their  terminology, (twisted) grid diagrams are \emph{nice} (multipointed, genus 1) Heegaard diagram for $\lp$, so the differential of $\mathrm{CFK}$ can be computed combinatorially. In this setting, the holomorphic disks of knot Floer homology's differential take the milder form of embedded rectangles on the grid.

Consider two generators $x,y \in S(G)$ having the same $\spinc$ degree;
if the permutations associated to $x$ and $y$  differ by a transposition, then the two components where the generators differ are the vertices of four immersed  rectangles $r_1, \ldots,r_4$ in the grid; the sides of the $r_i$'s are alternately arcs on the $\alpha$ and $\beta$ curves.
We can fix a \emph{direction} for such a rectangle $r$, by prescribing that $r$ \emph{goes from $x$ to $y$} if its lower left and upper right corners are on $x$ components. 
Therefore, if $x$ and $y$ are two generators in the same $\spinc$ degree, and differing by a single transposition, there are exactly two directed rectangles from $x$ to $y$.

\begin{defi}
Given a grid $G$, and $x,y \in S(G)$, call $Rect(x,y)$ the set of directed rectangles connecting $x$ to $y$; we will denote by $$Rect(G) = \bigcup_{x,y \in S(G)} Rect(x,y)$$ the set of all directed rectangles between generators in $G$. Similarly, $Rect^\circ (G)$ is going to be the set of \emph{empty rectangles}, that is those $ r \in Rect(x,y)$ such that $ Int(r) \cap x = \emptyset$.\\
Note that by assumption if $r \in Rect^\circ(x,y)$, then it does not contain any point of $y$ either, so in particular it is embedded in the torus composed by the grid.
\end{defi}
\begin{figure}
\includegraphics[width=8cm]{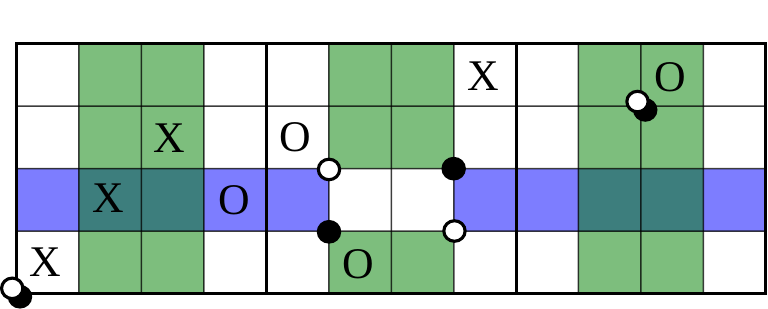}
\caption{Two directed rectangles connecting $x$ (white) to $y$ (black) in a grid of parameters (4,3,1). Only the horizontal one (in blue) is empty.}
\label{fig:rettangoli}
\end{figure}
One example of two rectangles is displayed in Figure~\ref{fig:rettangoli}.
If $x,y \in S(G)$, then $|Rect(x,y)|\in\{0,2\}$, and we will prove in Proposition~\ref{teo:grosso} that it can be non-zero only for generators in the same $\spinc$ degree which differ by a single transposition. On the other hand, with the same hypothesis on the generators, $|Rect^\circ (x,y)| \in \{0,1,2\}$.

If $r_1 \in Rect(x,y)$ and $r_2 \in Rect(y,z)$ we can consider their concatenation $r_1 \ast r_2$, which we call a \emph{polygon} connecting $x$ to $z$ through $y$.
We are going to denote by $Poly(x,z)$ the set of polygons connecting $x$ to $z$, and by $Poly^\circ(x,z)$ the empty ones.
 If $P$ is an empty rectangle or polygon, denote by $O_i(P) $
 the number of times (with multiplicity) that the $i$-\emph{th} $\mathbb{O}$ marking appears in $P$. Note that for a grid diagram of a knot in either $S^3$ or a lens space, we have $O_i(P)\in \{0,1,2\}$.
 
The differential on $\mathrm{GC}^-(G)$ and $\widehat{\mathrm{GC}}(G)$ is just going to be a count of empty rectangles, satisfying some additional constraints according to the flavour chosen.
For the two flavours of grid homology considered here (keep in mind that for $\hatc$ we set $V_1 = 0$) we keep track of the $\OO$ markings contained in the rectangles, by multiplying with the corresponding variable $V_i$:
\begin{equation}\label{eqn:differential}
\partial (x) = \sum_{y \in S(G)} \sum_{\substack{r \in Rect^\circ (x,y) \\ r \cap \XX = \emptyset} } \left(\prod_{i=1}^n V_i^{O_i(r)} \right) y
\end{equation}
\begin{prop}\label{teo:grosso}
Given a grid diagram $G$ of parameters $(n,p,q)$, the modules $\minusc (G)$ and  $\hatc (G)$ endowed with the endomorphism $\partial$ are chain complexes, that is $\partial^2 = 0$ in both cases. Moreover, $\partial$ acts on the tri-grading as follows:
\begin{enumerate}
\item $S(\partial (x)) = S(x)$
\item $M(\partial (x)) = M(x) -1$
\item $A(\partial (x)) = A(x)$
\end{enumerate}
\end{prop}
\begin{rmk}
This Proposition is implicit in \cite{BGH}, and it can be seen as a direct consequence of Theorem~$1.1$ therein; however some of the considerations in this proof will be useful in the following section. Moreover this proof will rely only on combinatorics, showing that the result can be obtained without any reference to the holomorphic theory of \cite{holdisknot} and \cite{rasmussenknot}.
\end{rmk}
\begin{proof}
We begin by examining the behaviour of the degrees under the differential; condition $(1)$ is easy to prove: by Equation~\eqref{eqn:spinc} the only relevant part of a generator $x$ for the computation of $S(x)$ is given by its $p$-coordinates. If $y$ appears in the differential of $x$, call $a_i^x,a_i^y, a_j^x$ and $a_j^y$ the $p$-coordinates where $x$ and $y$ differ. If $a_i^x = k$ and $a_j^x = l$, with  $k,l \in \zetap$, then (since by hypothesis they are connected by a rectangle) $a_i^y = k + t$ and $a_j^y = l-t$ modulo $p$ for some $t \in \{0,\ldots, p-1\}$ (as shown in Figure~\ref{fig:rectspin});  so $S(x)=S(y)$.\\
\begin{figure}[ht]
\includegraphics[width=7cm]{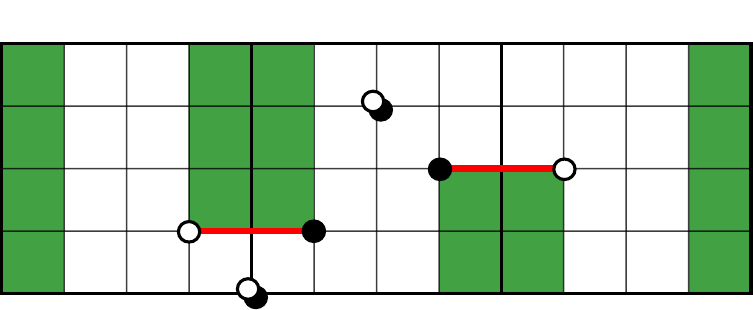}
\caption{The non equal $p$-coordinates of the generators compensate each other.}
 \label{fig:rectspin}
\end{figure}
Let's now see what happens for the other two degrees; if $x$ and $y$ are generators in $G$ connected by an empty rectangle $r$, directed from $x$ to $y$, then their lifts $\widetilde{x}$ and $\widetilde{y}$ will differ in $2p$ positions, according to the pattern suggested in Figure~\ref{fig:IXX}.
\begin{figure}[ht]
\includegraphics[width=7cm]{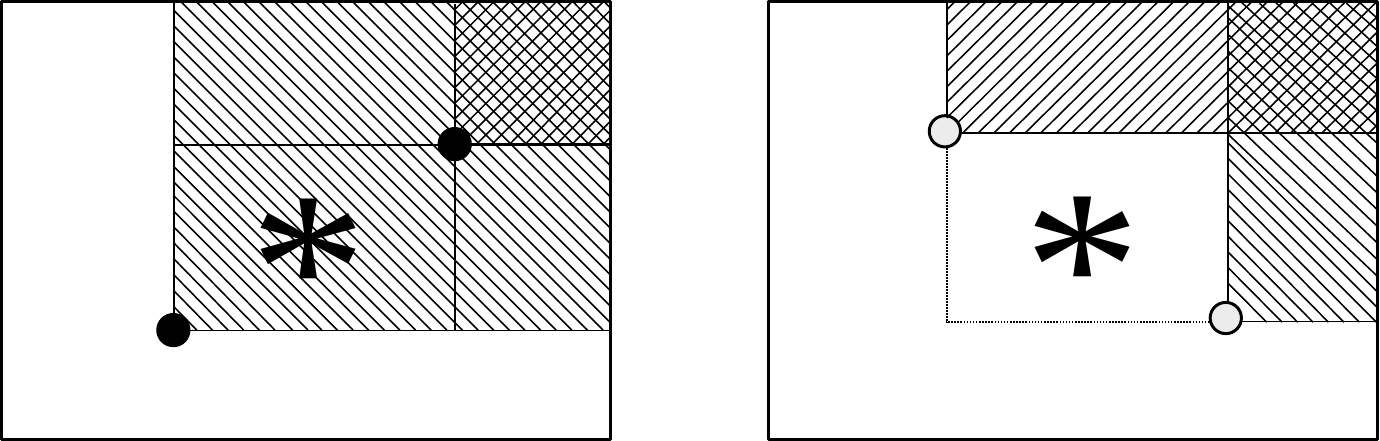}
 \caption{The difference between the functions $\iab{x}{*}$ and $\iab{y}{*}$ for two generators (in black and white respectively) whose permutations differ by a transposition. The shading indicates the upper-right regions considered in the computation of the function $\mathcal{I}$.}
 \label{fig:IXX}
\end{figure}
This implies that the corresponding $\mathcal{I}$ function will change accordingly:
$$\iab{\widetilde{x}}{\widetilde{x}} = \iab{\widetilde{y}}{\widetilde{y}} +p$$ 
$$\iab{\widetilde{x}}{\widetilde{\OO}} = \iab{\widetilde{y}}{\widetilde{\OO}} + p \sum_{i=1}^n O_i (r)$$
$$\iab{\widetilde{\OO}}{\widetilde{x}} = \iab{\widetilde{\OO}}{\widetilde{y}} + p \sum_{i=1}^n O_i (r)$$
Moreover, the same result holds with $\XX$ markings instead of $\OO$'s.
Then from Equation~\eqref{eqn:differential} we get for $(2)$ and $(3)$ respectively:
 $$\displaystyle M(\partial (x)) = \sum_{y \in S(G)} \sum_{\substack{r \in Rect^\circ (x,y) \\ r \cap \XX = \emptyset}} \left( \sum_{i=1}^n -2 O_i (r) \right) M(y)$$
 $$\displaystyle A(\partial (x)) = \sum_{y \in S(G)} \sum_{\substack{r \in Rect^\circ (x,y) \\ r \cap \XX = \emptyset}} \left(\sum_{i=1}^n - O_i (r)\right) A(y)$$
A substitution using Equations~\eqref{eqn:maslov} and \eqref{eqn:alex} defining the Maslov and Alexander degrees yields $(2)$ and $(3)$.\\

We are left to show that $ \partial^2 =0$; we thus need to study the possible decompositions in rectangles of polygons connecting two generators.\\
 We will prove the result for the minus flavoured complex, since the analogous result for the hat version follows immediately.
From Equation~\eqref{eqn:differential} we can compute 
\begin{equation}
\partial^2 (x) = \sum_{z \in S(G)} \sum_{\substack{\psi \in Poly^\circ (x, z) \\ \psi \cap \XX = \emptyset}}   N(\psi) \left( \prod_{i=1}^{n} {V_i}^{O_i (\psi)} \right) z
\end{equation}
where $\psi$ is a polygon connecting $x$ to $z$, and $N(\psi)$ is the number of possible ways of writing $\psi$ as the composition of two empty rectangles $r_1 \ast r_2$, with $r_1 \in Rect^\circ (x,y)$ and $r_2 \in Rect^\circ (y,z)$ for some $y\in S(G)$.

Note that a polygon $P$ connecting two generators is empty if and only if both the rectangles $P$ is made of are empty as well.
In order to complete the proof we need to show that $ N(\psi) \equiv 0 \;(\mbox{mod}\;2)$, \emph{i.e.}~there is an even number of ways (in fact exactly $2$) to decompose into rectangles a fixed $\psi$ that appears in the squared differential $\partial^2$.
We can also take advantage of the proof in \cite[Lemma $4.4.6$]{SOS} to reduce the number of cases to examine; as a matter of fact, if a polygon $\psi$
does not cross one of the $\alpha$ curves, we can cut the torus open along it, and think of the polygon as living in a portion of an $np\times np$ grid for $S^3$.
Thus, we only need to worry about polygons that intersect all $\alpha$ circles.\\

There are then four possibilities to be considered a priori, according to the quantity $M= |x\setminus (x \cap z)| \in \{ 0, \ldots, 4\}$, as schematically shown in Figure~\ref{fig:pallini}.
\begin{figure}[ht]
\labellist
\pinlabel $M=0$ at 65 10
\pinlabel $M=2$ at 320 10
\pinlabel $M=3$ at 570 10
\pinlabel $M=4$ at 825 10
\endlabellist

\includegraphics[width=8cm]{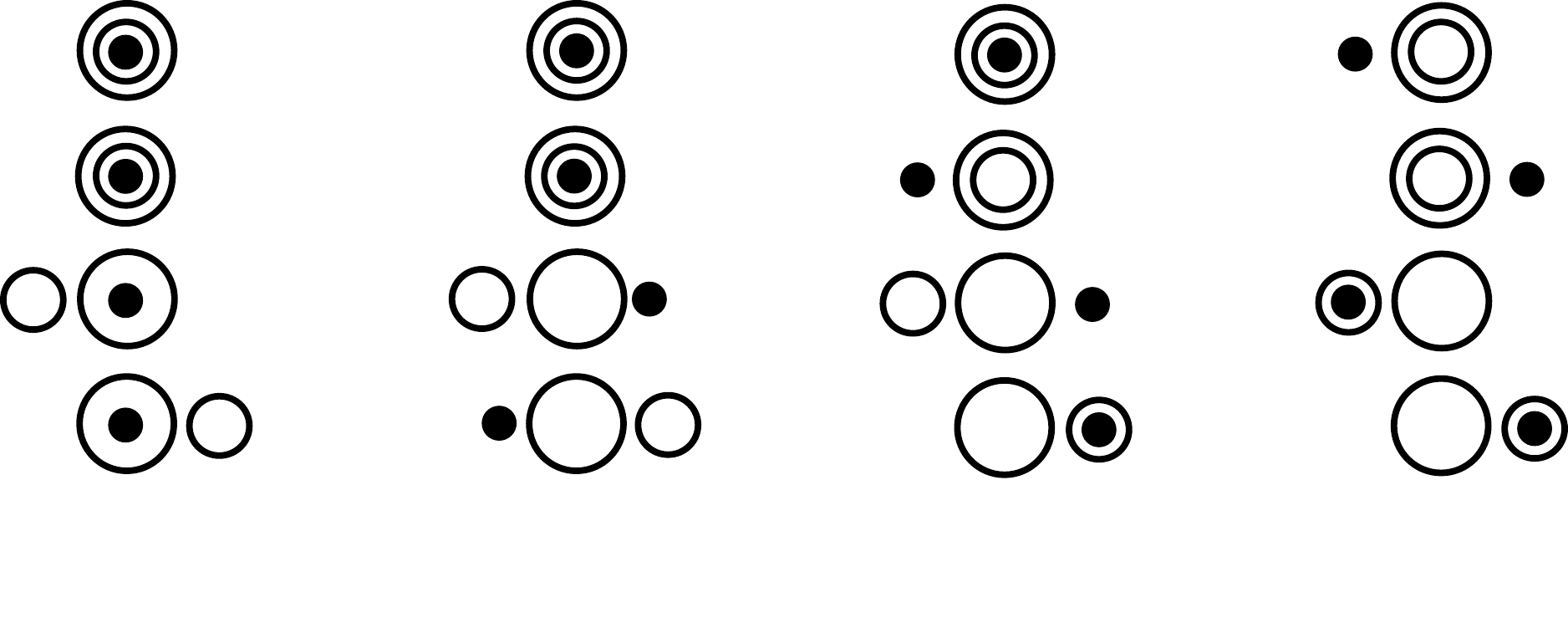}
\caption{A representation of the possibilities for $M$ (in a grid of dimension $4$). The circles correspond (from larger to smaller) to the components of generators $x,y$ and $z$. Two circles are concentric whenever the corresponding components of the generators coincide, and the $i$-th row shows the relative position for the components of the generators lying on $\alpha_i$.}
 \label{fig:pallini}
\end{figure}
If $M = 0$, that is $x = z$, the only possible polygons are \emph{thin} rectangles, called $\alpha$ and $\beta$ degenerations or simply $\alpha\slash \beta$-strips (see \emph{e.g.}~Figure~\ref{fig:alphadeg}).
These are strips of respectively height or width $1$ (otherwise they would not be empty). We are not concerned with these strips, since  each of them contains exactly one $\XX$ marking, hence they do not contribute to the differential.

This is not true for the \emph{filtered} versions of these complexes (see \cite[Ch.~13]{SOS}). Nonetheless, the polygons that cannot be split in two different ways cancel each other out nicely in that case as well. As an aside, we note here that there is only one way to decompose such a strip into two rectangles (one starting from $x$, and one arriving to it). 

The case $M = 1$ can be dismissed too,  since rectangles only connect generators which differ in exactly two points\footnote{And a product of two nontrivial and distinct transpositions is never a transposition.}.

If $M = 4$, that is, the corners of the two rectangles are all distinct, we can apply the same approach of \cite[Ch.~4]{SOS};
there are two ways of counting them, as shown in Figure~\ref{fig:rettangolidis}.
\begin{figure}[ht]
\includegraphics[width=12cm]{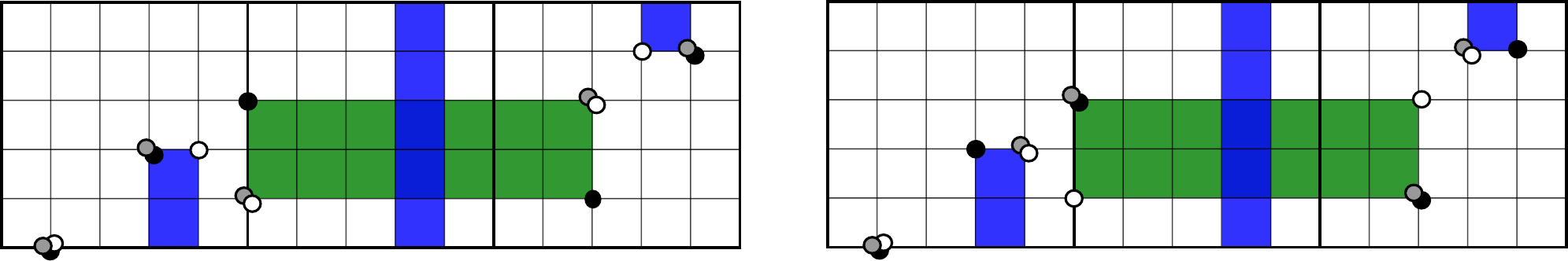}
\caption{When $M=4$ we can consider the two rectangles (from white to black) in either order, by choosing a suitable intermediate generator $y$ (gray).}
\label{fig:rettangolidis}
\end{figure}
Basically, the two decompositions correspond to taking the two rectangles in either order. We remark that one rectangle might wrap around the other, but the number of decompositions does not depend on this wrapping.

The  case $M=2$ needs a bit more care, since it has no $S^3$ counterpart (see \cite[Ch.~4]{SOS}). In this case, the two rectangles must share part of 2 edges.
There are two possibilities to consider:
\begin{enumerate}
\item the rectangle starting from $x$ does not cross all $\alpha$ curves. Up to vertical/horizontal translations, it can be placed in such a way that it does not intersect the boundary of the planar grid.
\item the rectangle starting from $x$ intersects all the $\alpha$ curves at least once.
\end{enumerate}
Either way, the second rectangle joining the intermediate generator to $z$ must end and start on the same $\alpha$ curves of the first rectangle; the configurations in both cases are shown Figure~\ref{fig:iugualea2}, together with their decompositions.\\

Lastly, if $M=3$ we can again distinguish  two possibilities as in the previous case; the combinatorially inequivalent configurations are shown in Figure~\ref{fig:iugualea3}, together with their two decompositions.
\begin{figure}[ht]
\includegraphics[width=12cm]{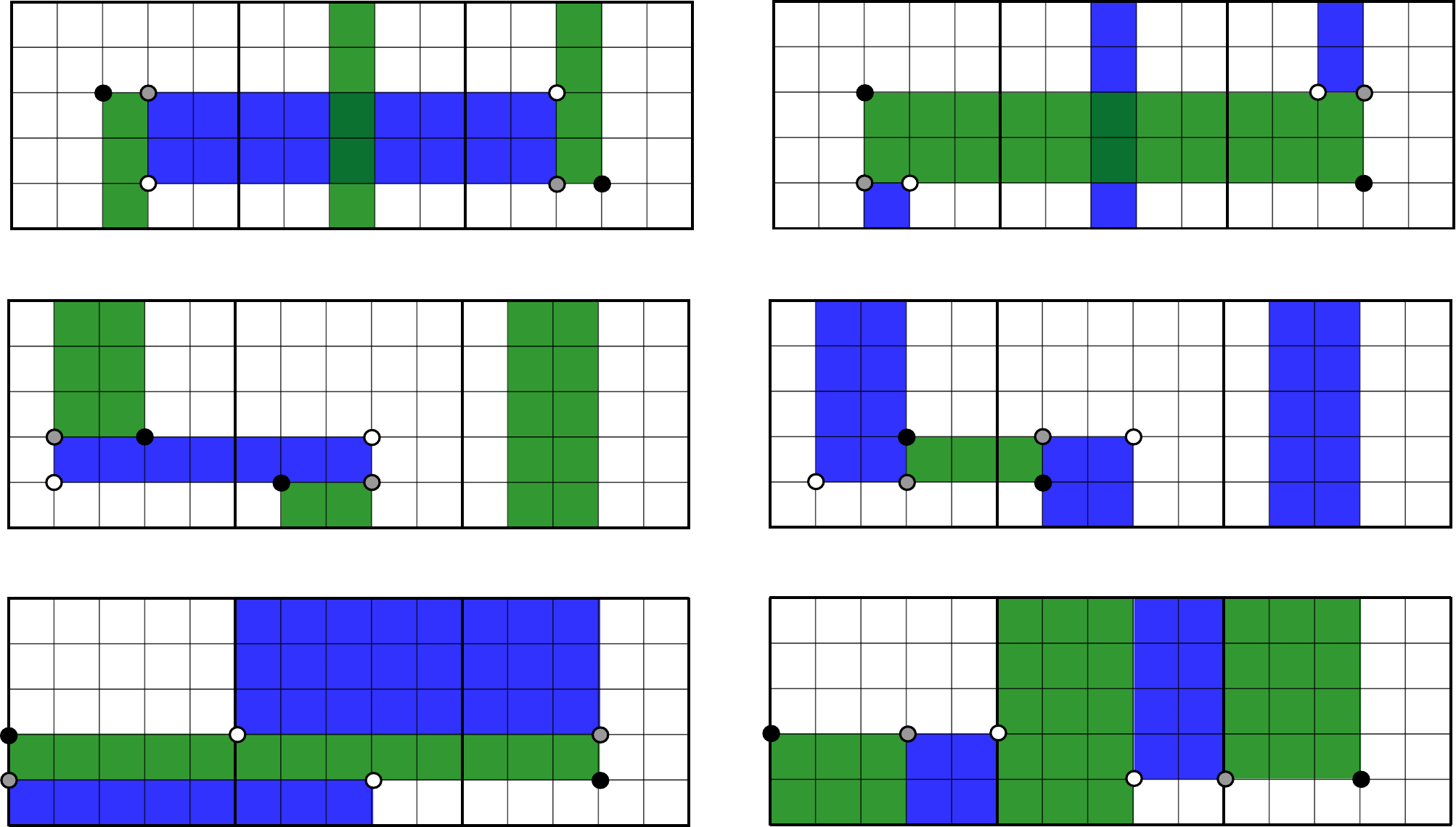}
\caption{Relevant combinatorial possibilities for $M=2$ on a grid for $L(3,1)$. On each row, the two possible decompositions are shown. Again we adopt the convention $x,y,z$ = white, grey and black dots, showing only the appropriate components.}
\label{fig:iugualea2}
\end{figure}
\begin{figure}[ht]
\includegraphics[width=11cm]{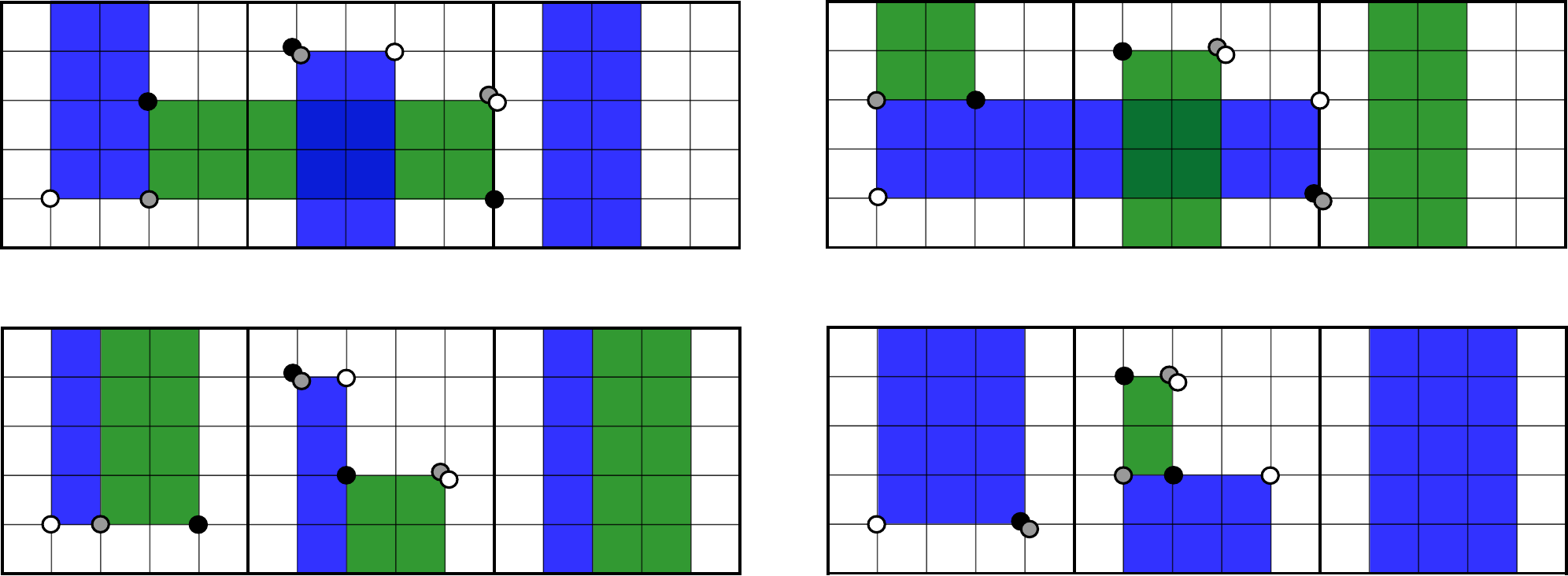}
\caption{Some configurations for the $M=3$ case. The complete combinatorial classification up to wrapping is presented in Figure~\ref{fig:combinatorialsign}.}
\label{fig:iugualea3}
\end{figure}
\end{proof}
\begin{ex}
We continue here the computations of the grid from Example~\ref{esempio}: we can now complete the picture by adding the differentials and computing the various homologies.
We have (\emph{cf}.~Figure~\ref{fig:esempiodiff}):
\begin{center}
\begin{itemize}
\item[] $\partial (a) = \partial (b) = 0$
\item[] $\partial (c) = a+b$
\item[] $\partial (d) =  V_1a + V_2 b$
\end{itemize}
\end{center}
\begin{figure}[ht]
\labellist
\pinlabel $V^2d$ at 125 50
\pinlabel $V^2d$ at 825 50

\pinlabel $Vd$ at 285 180
\pinlabel $Vd$ at 975 180

\pinlabel $V^2c$ at 130 180
\pinlabel $V^2c$ at 820 180

\pinlabel $V^2a$ at 0 210
\pinlabel $V^2a$ at 685 210

\pinlabel $V^2b$ at 0 160
\pinlabel $V^2b$ at 685 160

\pinlabel $d$ at 465 320
\pinlabel $d$ at 1150 320

\pinlabel $c$ at 465 450
\pinlabel $c$ at 1150 450

\pinlabel $Vc$ at 295 317
\pinlabel $Vc$ at 985 317

\pinlabel $Va$ at 180 347
\pinlabel $Va$ at 860 347

\pinlabel $Vb$ at 180 295
\pinlabel $Vb$ at 860 295

\pinlabel $a$ at 370 475
\pinlabel $a$ at 1060 475

\pinlabel $b$ at 370 435
\pinlabel $b$ at 1060 435

\pinlabel $A$ at 468 600
\pinlabel $V$ at 1230 535

\pinlabel $A$ at 1153 600
\pinlabel $V$ at 535 535
\endlabellist
\includegraphics[width=12cm]{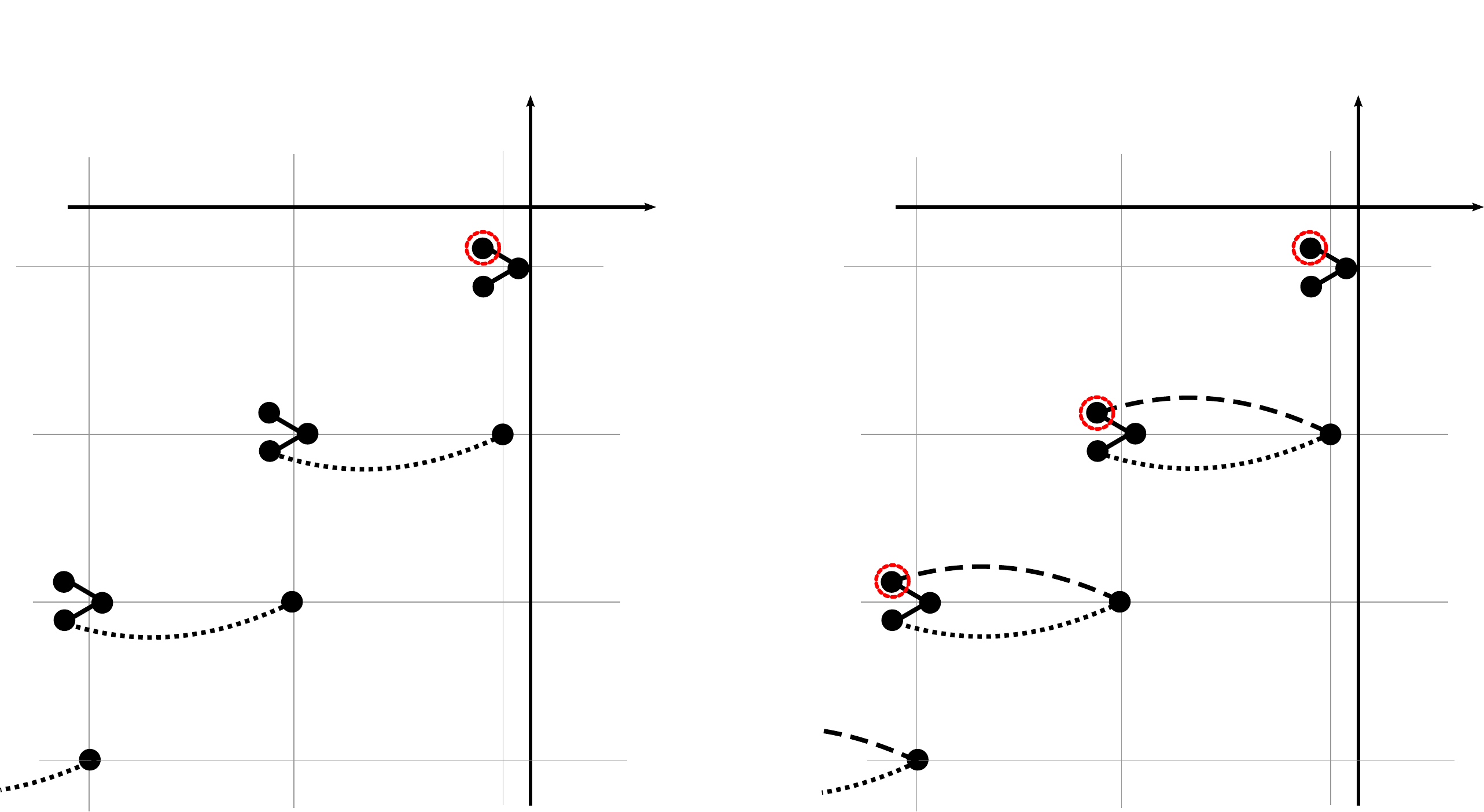}
\caption{On the left the complex $\hatc (G,0)$ and on the right the complex $\minusc (G,0)$, for the grid $G$ from Example~\ref{esempio};  the dotted line corresponds to multiplication by $V_2$, and the dashed one to multiplication by $V_1$. Non-trivial elements in homology are circled in red.}
\label{fig:esempiodiff}
 \end{figure}
It is then an easy task to compute the grid homologies in the two flavours:
$$\hath (G,0) = \F \left<a \right> \cong \F_{\left[-\frac{1}{4},-\frac{1}{4}\right]}$$
$$\minush (G,0) = \F[V_1] \left<a \right> \cong \F[U]_{\left[-\frac{1}{4},-\frac{1}{4}\right]}$$
\end{ex}

\subsection{The homologies}
From the definitions given up to now it might seem strange that the homology of such a complex could be an invariant of  the smooth isotopy type of a knot, since even the ground ring depends on the dimension of a grid representing it; Theorem \ref{teo:equiv} below ensures however that $\minush$ and $\hath$ are quasi-isomorphic to a finitely generated $\F [U]$ and $\F$-modules respectively (see \emph{e.g.}~Figure~\ref{fig:esempiodiff}). The algebraic reason behind this is the content of the following Proposition:
\begin{prop}
Let $G$ be a grid of parameters $(n, p,q)$ for a knot $K$. Then the action of multiplication by $V_i$ on the complex $\minusc (G)$ is quasi-isomorphic to multiplication by $V_j$.
\end{prop}
\begin{proof}
See \cite[Ch.~4]{SOS}.
\end{proof}
\begin{thm}[Thm.~1.1 in \cite{BGH}]\label{teo:equiv}
The homologies  $$ \minush (G)= H_*\left(\minusc (G) , \partial \right)$$ and $$\hath (G) = H_*\left(\hatc (G) , \partial \right)$$ regarded as $(\Q,\Q,\zetap)$ graded modules over the appropriate ring are invariants of the knot $(\lp , K)$.
Moreover, $\left(\minusc (G) , \partial \right)$ is quasi-isomorphic to a finitely generated $\F [U]$ module, where $U$ acts as any of the $V_i$, and $(\hatc , \partial )$ is quasi-isomorphic to a finitely generated $\F$ module.
\end{thm}

Due to this theorem, we will sometimes make the notational abuse of writing $\hath (\lp , K)$ instead of $\hath (G)$, $G$ being a grid of parameters $(n,p,q)$ representing $K$.
\begin{rmk}
Since the differential preserves the decomposition of the complex in $\spinc$ structures (Proposition~\ref{teo:grosso}), we can write 
$$\minush (\lp,K) = \bigoplus_{s\in \zetap} \minush (\lp,K,s)$$
$$\hath (\lp,K) = \bigoplus_{s\in \zetap} \hath (\lp,K,s)$$
and according to Proposition~\ref{teo:grosso} the endomorphism $U$ induced in homology by any of the $V_i$ acts as
$$\funct{U}{\minush_m (\lp,K,s,a)}{\minush_{m-2} (\lp,K,s,a-1)}$$
where $\minush_m (\lp,K,s,a)$ is the homology in tri-grading $(m,a,s)    \in (\Q,\Q,\zetap)$.
\end{rmk}
We can finally state the main result from \cite{BGH}:
\begin{thm}[Thm.~1.1 of \cite{BGH}]
Let $G$  be a grid for a knot $K \subset \lp$. There is a graded isomorphism of $\F[U]$ and respectively $\F$  tri-graded modules:
$$HFK^- (\lp,K) \cong \minush (G)$$
$$\widehat{HFK}(\lp,K) \cong \hath(G)$$
\end{thm}
\begin{rmk}\label{rmk:connectedsum}
Knot Floer homology is  known (see \emph{e.g.}~\cite[Thm.~7.1]{holdisknot}) to satisfy a formula\footnote{We do not specify here the various conventions involved for $\spinc$ structures in general, since in what follows we will only deal with $Y = S^3,\lp$, see~\cite[Sec.~7]{holdisknot} for a detailed account.} for the  connected  sum of two knots in rational homology 3-spheres; if  $ (Y,K) = (Y_0,K_0)\#(Y_1,K_1)$, then
\begin{equation}\label{eqn:connectedsum}
HFK^-\left( Y,K,s\right) \cong \bigoplus_{s_0 + s_1 = i} HFK^-\left( Y_0 , K_0 ,s_0\right) \otimes_{\F[U]} HFK^-\left(Y_1, K_1 ,s_1\right) 
\end{equation}
In each connected 3-manifold $Y$ the isotopy class of the homologically trivial unknot $\bigcirc$ is unique (since it bounds an embedded disk and manifolds are homogeneous); thus we can think of a \emph{local} knot $K$, \emph{i.e.}~a knot contained in a 3-ball inside $Y$ as the connected sum 
$$(Y,K) = (Y,\bigcirc) \# (S^3 , K^\prime )$$
for some knot $K^\prime$ in $S^3$.
It is then a straightforward computation to show that the grid homology of the unknot $\bigcirc \subset \lp$ is:
$$\minush (\lp, \bigcirc) = \bigoplus_{s \in \zetap} \F[U]_{\left[d(p,q,s),0\right]}$$
$$\hath (\lp, \bigcirc) = \bigoplus_{s \in \zetap} \F_{\left[d(p,q,s),0\right]}.$$
So, by Equation~\eqref{eqn:connectedsum}
$$\minush(\lp,K) = \bigoplus_{s \in \zetap}  \minush (S^3,K^{\prime})_{\left[d(p,q,s),*\right]}.$$
In other words, the grid homology of a local knot is completely determined by the homology of the same knot viewed as living in $S^3$ (and in particular its Alexander degrees are integers).
\end{rmk}
\begin{rmk}\label{rmk:semplici}
A \emph{simple knot} in $\lp$ is a knot admitting a grid of dimension 1, (see also \cite{hedden2011floer} and \cite{rasmussen2007lens}).
It is easy to show that in each lens space $\lp$ there is exactly one simple knot in each homology class; for $m \in \mathrm{H}_1 (\lp;\Z)$, denote this knot  by $T^{p,q}_m$. If $G$ is the dimension $1$ grid representing a simple knot in $L(p,q)$, then $|S(G)| = p$, and there is exactly one generator in each $\spinc$ degree. Therefore, there is no differential (since $\partial$ preserves the $\spinc$ degree), so the homologies of $T^{p,q}_m$ are:
$$\minush (T^{p,q}_m) \cong \bigoplus_{s\in \zetap} \F[U]_{[d(p,q,s),A(x_s)]}$$
$$\hath (T^{p,q}_m) \cong \bigoplus_{s\in \zetap} \F_{[d(p,q,s),A(x_s)]}$$
where $A(x_s)$ is the Alexander degree of the unique generators in degree $s$.
As in \cite{BGH} we say that these knots are Floer simple (or $U$-knot in the terminology  of \cite{absolutely}), meaning that the rank of the grid homology (over the appropriate ground ring) is exactly one in each $\spinc$ degree.
\end{rmk}

\section{Lift to $\Z$ coefficients}\label{sec:zcoeff}
The complexes we have used until now were defined to work with $\F$ as base ring; in particular, the proof of Proposition \ref{teo:grosso} relied on the parity of polygon decompositions to ensure that  $\left( \minusc , \partial \right)$ is in fact a chain complex. This section is devoted to prove Theorem~\ref{thm:main}, by exhibiting a combinatorial extension of the previous construction to integer coefficients. This was first done in the combinatorial setting for $S^3$ in \cite{manolescu2007combinatorial} (see also \cite{ozsvath2014combinatorial}).\\ We will adopt the group theoretic approach first developed in \cite{gallais2008} to define a sign function on rectangles, whose properties are precisely tuned to have $ \partial^2=0$. We note here that knot Floer homology can be defined in the analytic setting with integer coefficients \cite{holdisknot}, but it is currently not known whether it  coincides with its combinatorial counterparts.

A natural question to ask is to what extent the theory can change under such a change of coefficients; at the time of writing, there is no example of a knot in $S^3$ whose knot Floer homology with $\Z$ coefficients exhibits torsion (see Problem 17.2.9 of \cite{SOS}). 
Even in the lens space case, the computations displayed in section \ref{sec:computations} seem to suggest an analogous situation (see also the discussion in \cite[Sec.~15.6]{SOS}).\\ 

It is convenient to define signs on $Rect(G)$, rather than directly on $Rect^\circ (G)$; moreover the signs will not depend on the choice of a knot, but just on the parameters of the grid.
\begin{defi}
Given a grid diagram $G$, a \emph{sign assignment} on $G$ is a function $$\funct{\mathcal{S}}{Rect(G)}{\{ \pm 1\}}$$
such that the following conditions hold:
\begin{enumerate}
\item If $r_1 \ast r_2 = r_3 \ast r_4$ then $\segno{r_1} \segno{r_2} = - \segno{r_3}\segno{r_4}$
\item If $r_1 \ast r_2$ is a horizontal annulus ($\alpha$-strip), then $\segno{r_1}\segno{r_2} = 1 $
\item If $r_1 \ast r_2$ is a vertical annulus  ($\beta$-strip), then $\segno{r_1}\segno{r_2} = -1 $
\end{enumerate}
\end{defi}
We will prove in Theorem \ref{teo:segni} that sign assignments actually exist on twisted grid diagrams, and deal with problems related to their uniqueness later on.\\

We can now show how such a sign $\mathcal{S}$ can be used to promote $\hatc(G) $ and $ \minusc(G)$
from $\F \left[ V_1, \ldots, V_n \right]$ to $\mathbb{Z}\left[ V_1, \ldots, V_n \right]$ complexes.

To see why the properties given in the previous definition are indeed the right ones, fix a sign assignment $\mathcal{S}$ for $G$, and define
$$\partial_\mathcal{S} (x) =  \sum_{y \in S(G)} \sum_{\substack{r \in Rect^\circ (x, y)\\ r \cap \XX = \emptyset}} \mathcal{S}(r) \left( \prod_{i=1}^{n} {V_i}^{O_i (r)} \right) y , $$
Let us examine the coefficient of a generator $z \neq x$ in $ \partial_{\mathcal{S}}^2 (x)$; each polygon connecting $x$ to $z$ can be decomposed in two ways (as seen in Proposition~\ref{teo:grosso}).
The pairs corresponding to inequivalent decompositions of the same polygon cancel out due to condition (1) on $\mathcal{S}$.\\

If instead $x = z$ there are exactly $2n$ possible ways of connecting a generator to itself with empty polygons, which are $\alpha$ and $\beta$ degenerations; as noted before all of these strips contain one $\XX$ marking, so they do not contribute to the differential.\\
In order to prove the existence of a sign assignment, we are going to adopt the approach used in
\cite{SOS}, which relies on the paper \cite{gallais2008} of Gallais regarding the so-called \emph{Spin extension} of the permutation groups, introduced in the next definition.
\begin{defi}\label{def:spinext}
The \emph{Spin central extension} of the symmetric group $\mathfrak{S}_n$ is the group $\generalperm$ generated by the elements
$$\left< z, \taugen{i}{j} \: | \: 1\le i \neq j \le n \right>$$
subject to the  following relations:
\begin{itemize}
\item $z^2 = 1\:$ and  $\: z \taugen{i}{j} = \taugen{i}{j} z\:$ for $\:1 \le i \neq j \le n$
\item $ \taugen{i}{j}^2 =z\: $ and $\:\taugen{i}{j} = z\taugen{j}{i}  $
\item $ \taugen{i}{j} \taugen{k}{l} = z \taugen{k}{l} \taugen{i}{j} \:$ for distinct $\:1 \le i,j,k,l \le n$
\item $\taugen{i}{j} \taugen{j}{k} \taugen{i}{j} = \taugen{j}{k} \taugen{i}{j} \taugen{j}{k} = \taugen{i}{k}\:$ for distinct $\:1 \le i,j,k \le n$
\end{itemize}
\end{defi}
\begin{rmk}
The name Spin central extension is justified by the fact that this group can be derived as a $\faktor{\mathbb{Z}}{2\mathbb{Z}}$ extension of $\mathfrak{S}_n$ induced by the short exact sequence
\begin{equation}\label{eqn:seccs}
1 \longrightarrow \faktor{\mathbb{Z}}{2\mathbb{Z}} \longrightarrow  \generalperm  \xrightarrow{\mbox{  }\pi\:\mbox{  }} \mathfrak{S}_n \longrightarrow 1.
\end{equation}
Here $\pi$ is the surjective homomorphism defined by $\pi(z) =1$ and $\pi(\taugen{i}{j}) =\tau_{i,j}$.
\end{rmk}
\begin{defi}\label{def:section}
A section for $\generalperm$ is a map $$\funct{\rho}{\mathfrak{S}_n}{\generalperm}$$ such that $\pi\circ \rho = \mathrm{Id}_{\mathfrak{S}_n}$.
We will make a slight notational abuse, and also call sections the maps
$$\funct{\rho}{\mathfrak{S}_n \times \zetap^n }{\generalperm \times \zetap^n }$$
given by considering the product of a section with the identity map on $\zetap^n$.
\end{defi}
We are going to define a map
\begin{equation}
\varphi : Rect(G) \longrightarrow \widetilde{\mathfrak{S}}_n \times \zetap^n
\end{equation}
that associates to a rectangle $r \in Rect(x,y)$ an element in $ \widetilde{\mathfrak{S}}_n \times \zetap^n$,
enabling us to ``compare'' the generators that compose the vertices of $r$. If the elements of $x$ and $y$ in the bottom edge of $r$  belong respectively to $\beta_i$ and $\beta_j$, the
first component of $\varphi(r)$ is given by the generalised transposition $\widetilde{\tau}_{i,j}$. The second component of $\varphi$ is given by the difference
between the $p$-coordinates of $x$ and $y$.
The two generators differ only in two components, so necessarily $$(a^x_1 - a^y_1, \ldots, a^x_{n} - a^y_{n}) = (0,\ldots, 0, \pm k, 0,\ldots, 0, \mp k, 0, \ldots,0)$$
for some $k \in \{0, \ldots, p-1\}$.
\begin{rmk}
To simplify the proof of the next theorem, we observe here that the generalised permutation part of the map $\varphi$ does not depend on the eventual ``wrapping'' of a rectangle on the grid, while the $\zetap^n$ part does.
\end{rmk}
\begin{figure}[ht]
\labellist
\pinlabel $i$ at 0 10
\pinlabel $j$ at 105 10
\pinlabel $i$ at 295 10
\pinlabel $j$ at 386 10
\endlabellist
\includegraphics[width=5cm]{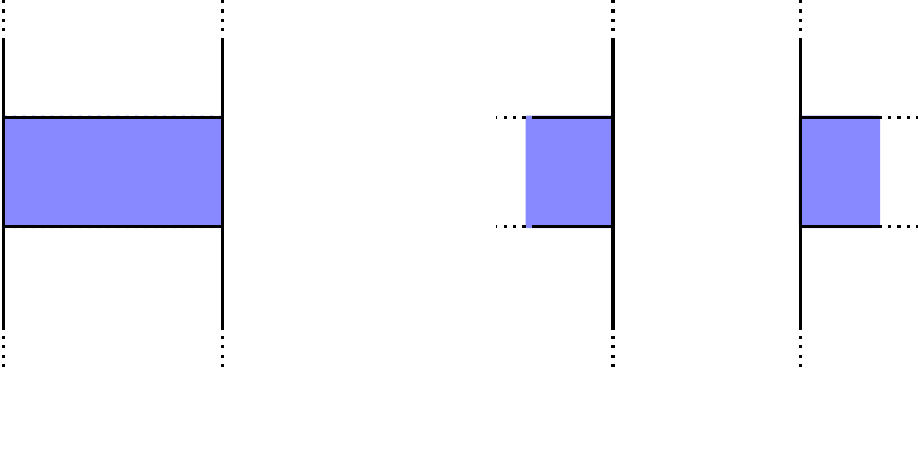}
\caption{The generalised transpositions associated to these two rectangles are $\widetilde{\tau}_{ij}$ and $\widetilde{\tau}_{ji} = z\widetilde{\tau}_{ij}$.}
\label{fig:traspogener}
 \end{figure}
\begin{ex}
Consider the rectangles $R$ in the left part of Figure~\ref{fig:rettangolidis}; the value $\varphi(R)$ associated is
 $\left( \widetilde{\tau}_{1,3}, (0,-1,0,1,0) \right)$ for the horizontal one and $\left( \widetilde{\tau}_{4,5}, (0,0,0,0) \right)$ for the vertical. Note also that swapping the order of the column's indices changes the associated generalised transposition, as indicated in Figure~\ref{fig:traspogener}.
\end{ex}
Given a section $\rho $ we can build a sign assignment as follows:
\begin{equation}\label{eqn:segno}
\mathcal{S}_\rho (r) =
\bigg \{
\begin{array}{rl}
1 & \mbox{ if } \rho(x) \varphi (r) = \rho (y) \\
-1 & \mbox{ if }\rho(x) \varphi (r) = z \rho (y)\\
\end{array}
\end{equation}
for $r \in Rect(x,y)$. The operation on $ \widetilde{\mathfrak{S}}_n \times \zetap^n$
 consists in the product of permutations on the first factor, and addition on the $p$-coordinates.
\begin{thm}\label{teo:segni}
For any given section $\rho $ on $\generalperm$, the function $ \mathcal{S}_\rho $ defined above is a sign assignment.
\end{thm}
\begin{proof}

First we deal with $\alpha$-strips; suppose $R_1 \in Rect(x,y)$, $R_2 \in Rect(y,x)$ are such that $R_1*R_2$ is an $\alpha$-strip. Then we have 
\begin{center}
$\varphi(R_1) = \left(\taugen{i}{j} , (0,\ldots, k , -k , \ldots ,0) \right)$\\
$\varphi(R_2) = \left(\taugen{j}{i} , (0,\ldots,-k , k , \ldots , 0) \right)$
\end{center}
for some indices $i,j$ and integer $k$.
\begin{figure}
\includegraphics[width=6cm]{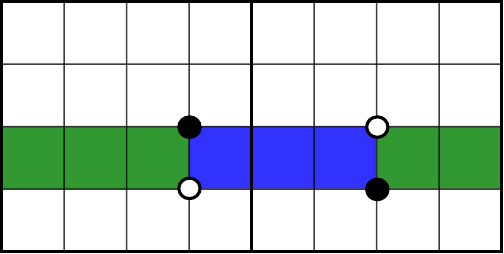}
\caption{An $\alpha$-strip of height 1. Only the relevant components of the generators are shown.}
\label{fig:alphadeg}
\end{figure}
So if $$\rho(x) \varphi(R_1) = \rho (y)$$
then
 $$\rho (x) = \rho(x) \varphi(R_1)\varphi(R_2) =  \rho (y) \varphi(R_2)$$
which implies $\mathcal{S}(R_1) = \mathcal{S}(R_1) = 1$.\\
If instead we had  $$\rho(x)\varphi(R_1) = z\rho (y)   \;\;\; \Rightarrow \mathcal{S}(R_1) = -1$$ then 
 $$z\rho (x) =  \rho (y) \varphi(R_2)  \;\;\; \Rightarrow \mathcal{S}(R_2) = -1$$
In both cases $\mathcal{S}(R_1) \mathcal{S}(R_2) = 1$.

Next, we examine the behaviour of signs for  $\beta$-strips.
As in the previous case, there is only one other possible generator $y$ that induces a decomposition of an annulus starting from $x$. The permutation components of the two rectangles $R_1 \in Rect(x,y)$ and $R_2 \in Rect(y,x)$  are both $\taugen{i}{j}$.
 So if $$\rho(x)\varphi(R_1) = \rho (y)  $$
 $$\rho(x) z =   \rho(x)\varphi(R_1)\varphi(R_2) = \rho (y)\varphi(R_2)$$ 
 which implies	 $\mathcal{S}(R_1) \mathcal{S}(R_2) = -1$.\\
The centrality of $z$ in $\generalperm$ tells us that the case with $\mathcal{S}(R_1) = -1$ is identical.\\

Now, given a general polygon $P = r*r^\prime$ connecting two generators $x \neq x^\prime$, we can easily prove that Definition \ref{def:spinext} implies  \begin{equation}\label{eqn:cambiosegno}
\rho(x^\prime) = z^{\frac{1 -\mathcal{S}(r) \mathcal{S}(r^\prime)}{2}} \varphi (r) \varphi (r^\prime) \rho (x).
\end{equation}
To show this, assume $r \in Rect(x,w)$ and $r' \in Rect(w,x')$ for some intermediate generator $w$. Then, by Equation~\eqref{eqn:segno}, it is immediate to see that $\varphi(r) \varphi(r') \rho(x)$ is equal to either $\rho(w) \varphi(r')$ or $z\rho(w) \varphi(r')$, according to whether $\mathcal{S}(r) =1$ or $\mathcal{S}(r) = -1$ respectively. In the same way, $\rho(w) \varphi(r')$ is equal to either $\rho(x')$ or $z \rho(x')$, this time according to the value of $\mathcal{S}(r')$. The factor $z^{\frac{1 -\mathcal{S}(r) \mathcal{S}(r^\prime)}{2}}$ in Equation~\eqref{eqn:cambiosegno} compensates the possible appearance of $z$ factors. \\

According to the proof of Proposition~\ref{teo:grosso}, each polygon which is not an $\alpha\slash \beta$-strip can be written as the concatenation of two distinct pairs of rectangles;
so we just need to check for all possible polygons
$$P = r(x,y) * r(y,x^\prime) = r(x,w)*r(w,x^\prime)$$ that the following identity holds:
\begin{equation}\label{eqn:segnidim}
\varphi(r(x,y)) \varphi (r(y,x^\prime)) = z \varphi(r(x,w)) \varphi (r(w,x^\prime)),
\end{equation}
where $y\neq w$ are two auxiliary generators which differ by only one transposition from $x$ and $z$.
All we need to do is verify Equation~\eqref{eqn:segnidim} in the cases $M = 2,3,4$ from the proof of Proposition~\ref{teo:grosso} (recall that we already considered $M=0$, and $M=1$ was discarded).\\

It is easy to check that the generalised permutations associated to polygons corresponding to the $M=3$ case are the same of~\cite[Ch.~15]{SOS} in the $S^3$ case; in particular this is true even when the rectangles wrap around the grid, since the generalised permutation part  does not depend on the $p$-coordinates of the generators.\\
The case $M=4$ is immediate: as shown in Figure~\ref{fig:segnidisgiunti}
\begin{figure}
\labellist
\pinlabel $k$ at 19 0
\pinlabel $i$ at 37 0
\pinlabel $j$ at 66 0
\pinlabel $i$ at 100 0
\pinlabel $j$ at 130 0
\pinlabel $i$ at 160 0
\pinlabel $j$ at 188 0
\pinlabel $l$ at 210 0
\endlabellist
\includegraphics[width=6cm]{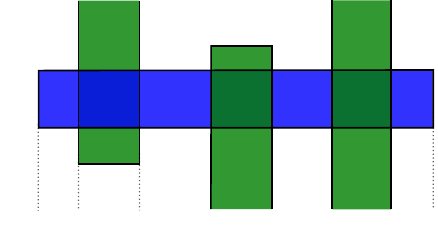}
\caption{The generalised permutations associated to the green and blue rectangles are $\widetilde{\tau}_{ij}$ and $\widetilde{\tau}_{kl}$ respectively.}
\label{fig:segnidisgiunti}
\end{figure}
the permutations associated to the two decompositions are such that Equation~\eqref{eqn:segnidim} becomes exactly the third relation defining $\generalperm$.\\
Lastly, we deal with $M=2$; the generalised transpositions associated to $r(x,y)$ and $r(y,x^\prime)$ are $\widetilde{\tau}_{ij}$ and $\widetilde{\tau}_{ji}$.
For the two rectangles $r(x,w)$ and $r(w,x^\prime)$ on the right in Figure~\ref{fig:iugualea2segni} the associated transposition is $\widetilde{\tau}_{ij}$ in both cases.
So in particular this implies that  if $\mathcal{S}(r(x,y))\mathcal{S}(r(y,x^\prime)) =-1$ then $\mathcal{S}(r(x,w))\mathcal{S}(r(w,x^\prime)) =1$ and vice versa. Therefore, Equation~\eqref{eqn:segnidim} is always satisfied, and we are done.
\begin{figure}[ht]
\includegraphics[width=8cm]{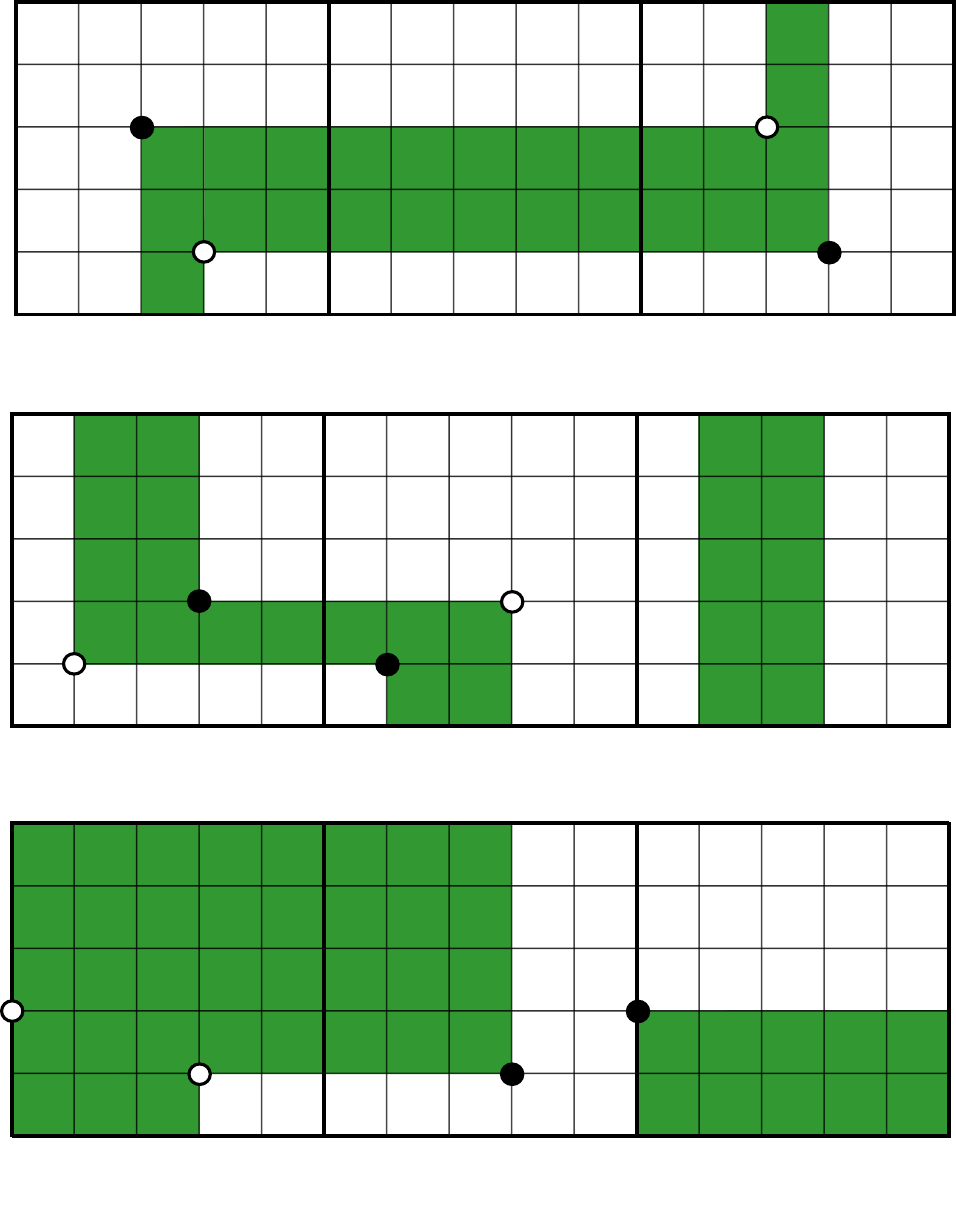}
\caption{The generalised transpositions $\widetilde{\tau}_{ij}$ and $\widetilde{\tau}_{ji}$ are associated to the two decompositions in the $M=2$ case.}
\label{fig:iugualea2segni}
\end{figure}
\begin{figure}[ht]
\labellist
\pinlabel $i$ at 8 10
\pinlabel $j$ at 95 10
\pinlabel $k$ at 187 10

\pinlabel $i$ at 294 10
\pinlabel $j$ at 381 10
\pinlabel $k$ at 473 10

\pinlabel $i$ at 8 270
\pinlabel $j$ at 95 270
\pinlabel $k$ at 187 270

\pinlabel $i$ at 294 270
\pinlabel $j$ at 381 270
\pinlabel $k$ at 473 270
\endlabellist
\includegraphics[width=7cm]{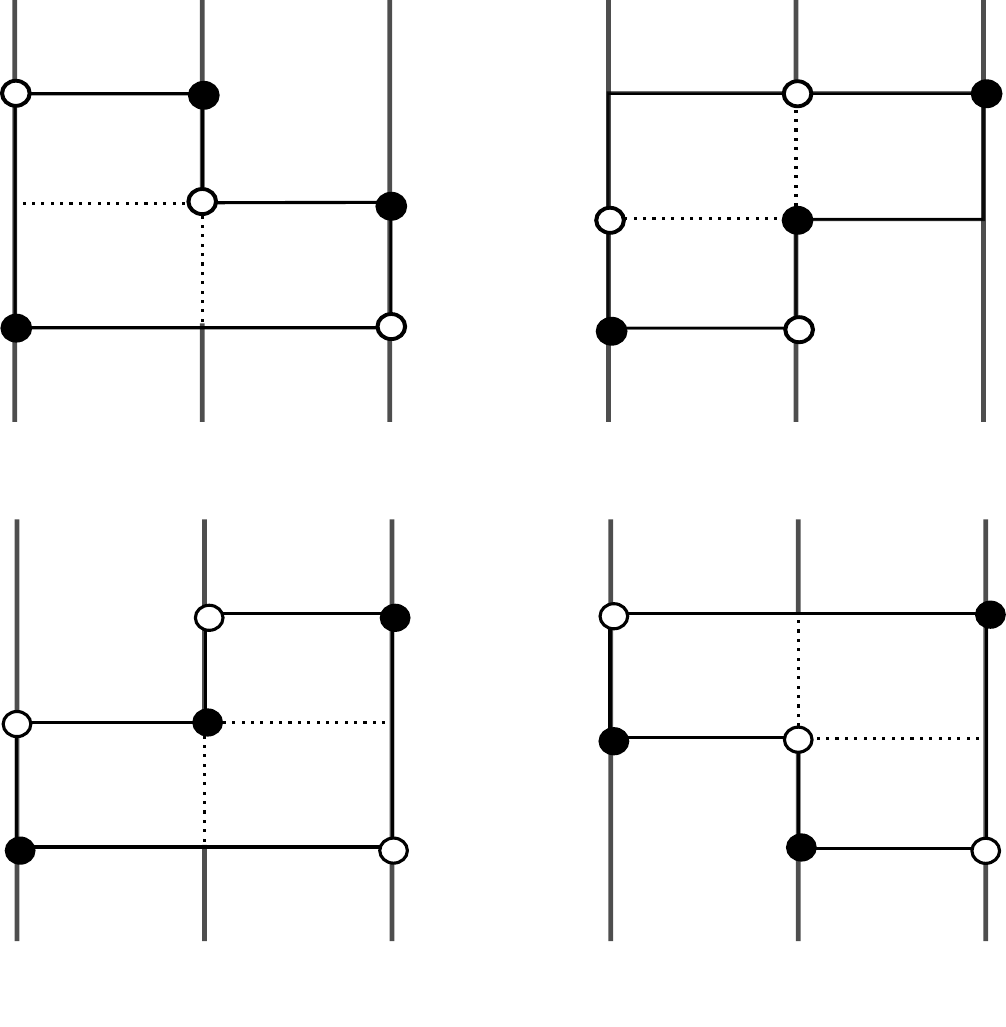}
\caption{The four relevant combinatorial possibilities for the $M=3$ case in the $S^3$ setting. Remember that the eventual wrapping of one rectangle over the other does not change the relations in $\generalperm$.}
\label{fig:combinatorialsign}
\end{figure}
\end{proof}
\begin{rmk}
It is worth noting that the trivial choice for signs --that is, treating each rectangle just as a generalised permutation, without keeping track of $p$-coordinates (like in the $S^3$ setting)-- can't distinguish a $\beta$ degeneration from other polygons which admit two distinct decompositions into rectangles, as shown in Figure~\ref{fig:stessaperm}.
\end{rmk}
\begin{figure}[ht]
\includegraphics[width=12cm]{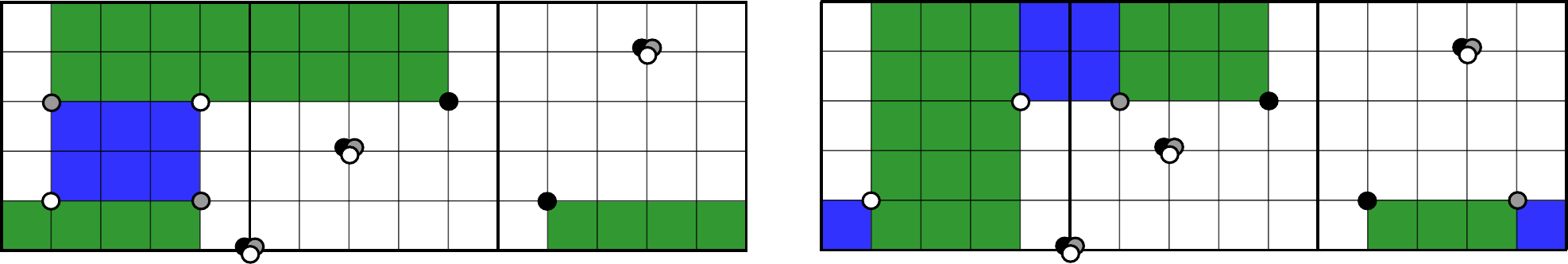}
\caption{The white and black generators have the same permutation component, but the polygon connecting them admits two distinct decompositions. In particular, it cannot be an $\alpha\slash\beta$-strip.}
\label{fig:stessaperm}
\end{figure}

\begin{defi}
A section $\funct{\rho}{\mathfrak{S}_n \times \zetap^n }{\generalperm \times \zetap^n }$ is said to be \emph{compatible} with a sign assignment $\mathcal{S}$ if for any pair $x,y$ of generators, and any $r \in Rect(x,y)$ 
$$\rho (y) =
\bigg \{
\begin{array}{rl}
\rho(x) \widetilde{\tau}(r) & \mbox{ if } \mathcal{S}(r) = 1 \\
z \rho(x) \widetilde{\tau}(r) & \mbox{ if } \mathcal{S}(r) = -1 \\
\end{array}
$$%\end{equation*}
is satisfied.
\end{defi}
\begin{prop}
Given a sign assignment $\mathcal{S}$, there are exactly two compatible sections.
\end{prop}
\begin{proof}
Sections were defined in Definition~\ref{def:section} to be extensions of algebraic sections of the short exact  sequence~\eqref{eqn:seccs} via the identity on the $\zetap^n$ component (for a grid of dimension $n$). Therefore, the statement is equivalent to proving the existence of exactly two compatible sections $\funct{\rho}{\mathfrak{S}_n}{\generalperm}$. This is precisely the content of \cite[Prop~15.2.13]{SOS}, so we can conclude.
\end{proof}

Now, for the uniqueness of sign assignments, denote by $\mathcal{G}auge(G)$ the group of maps $$v: S(G) \longrightarrow \faktor{\mathbb{Z}}{2\mathbb{Z}}$$ $\mathcal{G}auge(G)$ acts on sections as follows:
\begin{equation}
\rho^v (x) =
\bigg \{
\begin{array}{rl}
\rho(x)   & \mbox{ if } v(x) = 1\\
z\rho(x)  & \mbox{ if } v(x) = -1\\
\end{array}
\end{equation}

This action is free and transitive; $\mathcal{G}auge(G)$ also acts on the set of sign assignments:
if $S$ is a sign on a grid $G$ and $v \in \mathcal{G}auge(G)$, define  $S^v (r) = v(x) S(r) v(y)$ for $r \in Rect(x,y)$.

As in the $S^3$ case, it is easy to show that there is only one sign assignment on a grid, up to this action of $\mathcal{G}auge(G)$.
The uniqueness now follows by noting that if $S_1$  and $S_2$ are two sign assignments on a grid $G$, then $S_2 = S_1^v$ for some $v \in \mathcal{G}auge(G)$, and the map
$$f: (\minusc (G), \partial_{S_1}) \longrightarrow (\minusc (G), \partial_{S_2})$$ given by $f(x) = v(x) x$ is an isomorphism (of tri-graded $R$-modules).
This concludes the proof of Theorem~\ref{thm:main}.

\section{Computations}\label{sec:computations}
\subsection{The programs:}
It becomes immediately apparent that the work needed to actually compute $\hath (G)$ for grids with dimension greater than 3 is not manageable by hand\footnote{The generating set for a  grid with parameters $(n,p,q)$ has  $n! p^n$ elements!}. So the author developed several programs in Sage \cite{sage} capable of computing the hat flavoured grid homology of links in lens spaces.
The computation can be made with $\Z$ coefficients, provided that the grid dimension is less than 5.

By using this tool we were able to verify that all knots with a grid representative whose parameters satisfy the following conditions, are $r$-torsion free ($r\le 17$):
\begin{itemize}
\item for $n = 2$, $p \le 12$
\item for $n= 3$, $p \le 6$
\item for $n = 4$, $p \le 4$
\item $n = 5$, $p\le 2$
\end{itemize}
The programs can be freely used interactively online at my homepage \cite{homepage}, or downloaded and used on a local Sage distribution. We recall here that there are several programs that compute the grid homology\slash knot Floer homology of knots in the $3$-sphere. In particular M.~Culler's \emph{Gridlink} \cite{gridlink} includes code by J.A.~Baldwin and W.D.~Gillam \cite{baldwin2012computations} that computes $\widehat{\mathrm{GH}}$, and there is a more recent program by Ozsv\'ath and Szab\'o \cite{OSprog} that can quickly compute $\widehat{\mathrm{HFK}}$ for knots with relatively high crossing number.

\subsection{Grid homology calculator}
The input consists of the grid parameters $(n,p,q)$, followed by two strings of length $n$ determining the positions of the $\XX$ and $\OO$ markings. We encode the markings with a string of length $n$ for each kind; to the $i$-\emph{th} marking (from the bottom  row) we associate the number of the small square containing it (from the left, and starting from 0). As an example, the knot in Figure~\ref{fig:griglia} is encoded as $\XX = [12, 1,8,5,9]$ and $ \OO = [6,3,0,9,12]$.\\

The output consists of the following:
\begin{itemize}
\item (Optional) A drawing of the chosen grid
\item The hat grid homology\footnote{If the grid dimension is greater than 5 it returns the $\F$ version.} $\hath( G, \mathfrak{s};\Z)$ for each $\mathfrak{s} \in \spinc$ structure, and its decategorification.
\item Whether the knot is rationally fibred, the homology class and its rational genus (see \cite[Sec.~1]{ni2014heegaard} for the definitions).
\item (Optional) A long list of the generators with their bi-grading.
\item (Optional) A drawing of the grid for the lift of the knot to $S^3$, together with its (univariate) Alexander polynomial and the number of components of the lift.
\end{itemize}
Basically, the program creates the generators $S(G)$ and computes their degrees; afterwards it checks for empty rectangles, and creates the matrices of the differentials.

Rather than computing the module $\hatc (G)$, we adopt the simpler approach of computing yet another version of the grid homology, known as \emph{tilde} flavoured homology, $\widetilde{\mathrm{GH}}(G)$.

The complex is simply the free $\Z$ module generated over $S(G)$, and the differential counts only those empty rectangles that do not contain any marking:
$$\widetilde{\partial} (x) = \sum_{y \in S(G)} \sum_{ \substack{r \in Rect^\circ (x,y) \\ (\XX \cup \OO ) \cap r = \emptyset }} \mathcal{S}(r) y$$
where $\mathcal{S}$ is a sign assignment.\\
Using the handy group theoretic capabilities of Sage,  the relations in $\generalperm$ (for $n \le 5$) were encoded in a matrix associated to the differential.\\

A minor technical hurdle here is represented by the fact that the tilde flavoured version is not directly an invariant of the knot represented by the grid. This can be easily seen \emph{e.g.}~by computing $\tildeh (G)$ in any $\spinc$ degree, for the grids of Example~\ref{esempio}.
However, the hat version can nonetheless be recovered from it:
\begin{prop}[Prop.~4.6.13 of \cite{SOS}]\label{prop:tilde}
Given a grid $G$ of dimension $n$ representing the knot $K \subset \lp$, there is a graded isomorphism $$\tildeh (G) = H_* \left( \tildec (G),\widetilde{\partial} \right) \cong \hath (\lp, K) \otimes W^{\otimes(n-1) }$$
where $W = \Z_{[0,0]} \oplus \Z_{[-1,-1]}$.
\end{prop}
After computing the homology $\widetilde{\mathrm{GH}}(G)$, the program ``factors out" the tensor product dependent on the size of the grid, and prints the requested information.

\subsection{A small example}
Knot theory (and hence grid homology) in lens spaces is quite more complex than its 3-sphere counterpart: besides the fact that knots need not be homologically trivial, they also can be nontrivial for very small grid parameters. As an example, define
$$f(p) = \min\{\text{dimension of a grid representing a non-simple knot in } \lp \}$$
Then $f(1) = 5$ (this is the trefoil in $S^3$), $f(2) = 3$ and $f(p>2) = 2$ (for this last case see Figure~\ref{fig:0134}).
\begin{ex}\label{ex:computations}
We sketch here the computation for the various flavours of grid homology in the case of the knot in Figure~\ref{fig:0134}.  The generating set $S(G,s)$ in each $\spinc$ degree $s \in \{0,\pm1\}$ has 6 elements, which we will denote $x^0_s, \ldots, x^5_s$ for $s = 0,1$. We do not need to compute the $s = -1$ complex, as this is quasi-isomorphic to the $s=1$ one (this follows from the existence of an involutive operation on the set of $Spin^c$ structures, see \emph{e.g.}~\cite[Sec.~3]{holdisknot} for a detailed account).

\begin{figure}
\includegraphics[width=6cm]{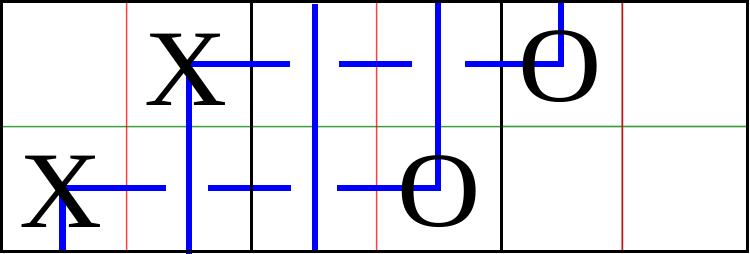}
\caption{The knot in $L(3,1)$ described by $\XX ,\OO = [0,1], [3,4]$.}
\label{fig:0134}
\end{figure}
We can now list the generators, with their bi-degree and differential:\\

\begin{tabular}{ccc}
generator & $\left(M, A \right)$ &  differential\\
\hline
$\spinc$  degree = 0 & & \\
\hline
& &\\
$x_0^0$ & $\left(\frac{3}{2}, 1 \right)$   & $\partial (x_0^0) = 0$\\
$x_0^1$ & $\left(\frac{1}{2}, 0 \right)$   & $\partial (x_0^1) = (V_1 - V_2) x_0^0$\\
$x_0^2$ & $\left(\frac{1}{2}, 0 \right)$   & $\partial (x_0^2) =  (V_2 - V_1) x_0^0$\\
$x_0^3$ & $\left(-\frac{1}{2}, -1 \right)$ & $ \partial (x_0^3) = V_2 \left( x_0^1 + x_0^2 \right)$\\
$x_0^4$ & $\left(-\frac{1}{2}, -1 \right)$ & $\partial (x_0^4) = V_1 \left( x_0^1 +  x_0^2 \right)$\\
$x_0^5$ & $\left(-\frac{3}{2}, -2 \right)$ & $\partial (x_0^5) = - V_1 x_0^3  +V_2 x_0^4$ \\
& &\\
\hline
$\spinc$  degree = 1 & & \\
\hline
& &\\
$x_1^0$ & $\left(\frac{7}{6}, 0 \right)$   & $ \partial (x_1^0) = -x_1^1 + x_1^2$\\
$x_1^1$ & $\left(\frac{1}{6}, 0 \right)$   & $\partial (x_1^1) = 0 $\\
$x_1^2$ & $\left(\frac{1}{6}, 0 \right)$   & $\partial (x_1^2) = 0$\\
$x_1^3$ & $\left(\frac{1}{6}, -1 \right)$  & $\partial (x_1^3) = x_1^4 - x_1^5$\\
$x_1^4$ & $\left(-\frac{5}{6},  -1\right)$ & $\partial (x_1^4) = V_2 x_1^1 - V_1 x_1^2$\\
$x_1^5$ & $\left(-\frac{5}{6}, -1 \right)$ &  $\partial (x_1^5) = V_2 x_1^1 - V_1 x_1^2 $   \\
&&\\
\end{tabular}\\
Since $\widetilde{\mathrm{GC}}(G, 0)$ has no differentials\footnote{Recall that its differential can be obtained by $\partial$, setting all $V_i$ variables to 0.}, the homology coincides with the complex. In $\spinc$ degree 1 instead, the tilde homology is generated by $x_1^1$ and $x_1^4$, so   $\widetilde{\mathrm{GH}}(G, 1) \cong \Z_{\left[\frac{1}{6},0 \right]} \oplus \Z_{\left[-\frac{5}{6},-1 \right]}$.\\ 

The computation of the minus flavour is just slightly more involved; $\minush (L(3,1) , K , 0)$ is composed by a copy of $\Z \left[U \right]$ generated by $x_0^0$, plus two $U$-torsion components, generated by $x_0^1 + x_0^2 $ and $x_0^3 + x_0^4 $. Altogether
$$ \minush (L(3,1) , K , 0)  = \Z [U]_{\left[\frac{3}{2}, 1 \right]} \oplus \Z_{\left[\frac{1}{2}, 0 \right]}  \oplus \Z_{\left[ -\frac{1}{2}, -1\right]} $$

In the last case, we get $$ \minush (L(3,1) , K , 1)  = \Z [U]_{\left[\frac{1}{6},0 \right]} $$
generated by $x_1^1$. The hat homology can be obtained either by factoring out the tensor product with $\Z_{\left[ 0,0\right]} \oplus \Z_{\left[-1,-1 \right]}$ from the tilde flavour, or deleting all dotted differentials in the minus complex of Figure~\ref{fig:complessi}, then computing the homology.
\begin{equation}
\hath (L(3,1),K,i) =
\bigg \{
\begin{array}{cl}
\Z_{\left[ \frac{3}{2} , 1 \right]} \oplus\Z_{\left[ \frac{1}{2} , 0\right]} \oplus \Z_{\left[ -\frac{1}{2} , -1\right]}  & \mbox{ if } $i = 0$\\
\Z_{\left[ \frac{1}{6},0\right]}  & \mbox{ if } i = \pm 1\\
\end{array}
\end{equation}
\end{ex}
This particular knot is interesting for several reasons: firstly, it is the smallest non-trivial or simple knot in any lens space. More importantly, it can be also proved that, despite being nullhomologous, it is not concordant (or even almost-concordant, see \cite{celoria2018concordances}) to a local knot. Finally, the group $\mathrm{GH}^-(G,0)$ is isomorphic (up to a shift in the Maslov grading) to the knot Floer homology of the trefoil knot in $S^3$; this in no accident, and a rather more general statement is \cite[Thm.~1.2]{upsilon}.

\begin{center}
\begin{figure}[ht]
\labellist
\pinlabel $A$ at 375 840
\pinlabel $M$ at 500 743
\pinlabel $A$ at 960 970
\pinlabel $V$ at 990 890
\pinlabel $A$ at 300 322
\pinlabel $M$ at 490 280
\pinlabel $A$ at 909 370
\pinlabel $V$ at 990 350
\endlabellist
\includegraphics[width = 13cm]{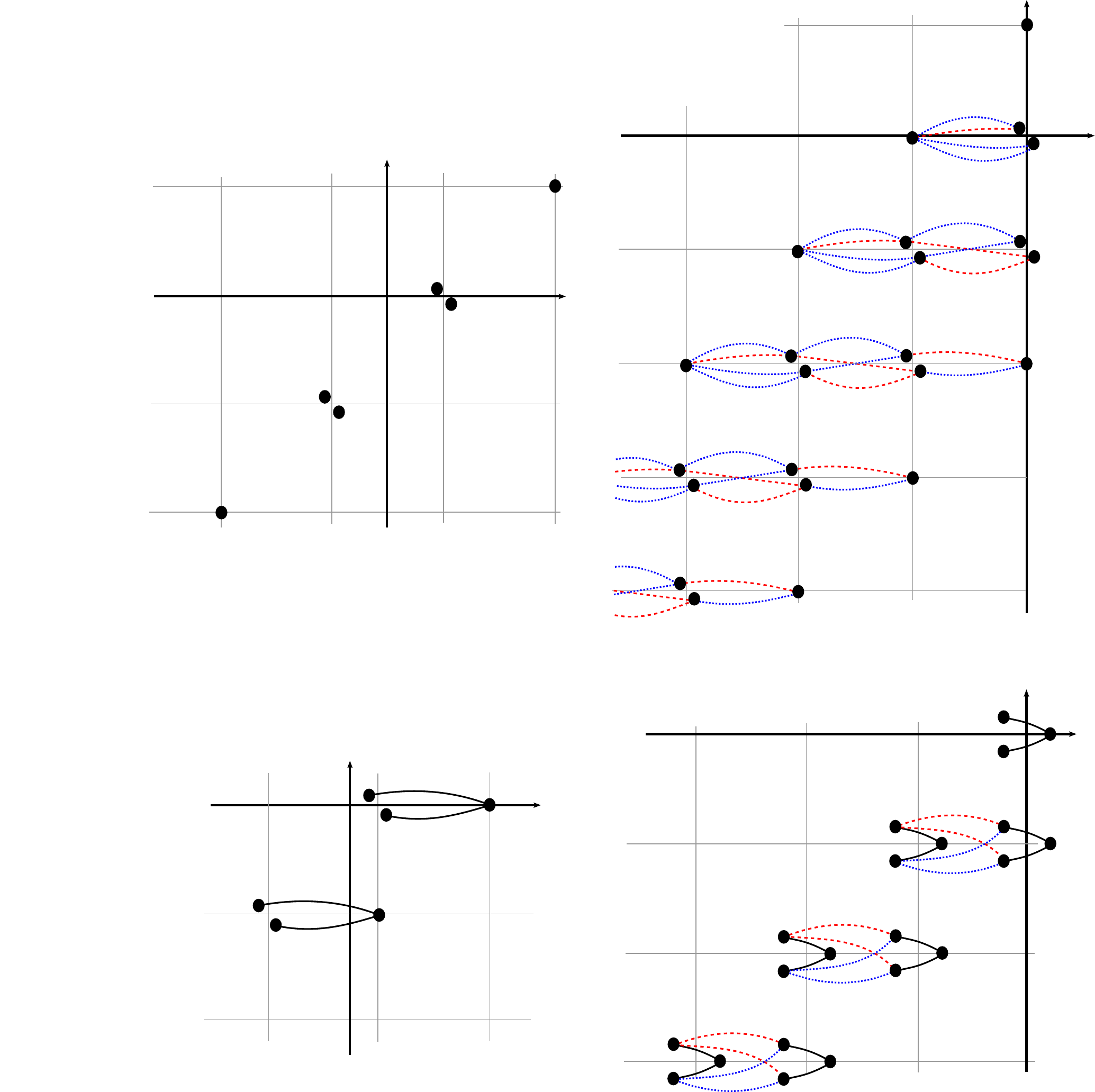}
\caption{The complexes $\widetilde{GC}(G, i)$ (on the left) and $\minusc (G , i)$  (on the right) for $i = 0,1$. Red dashed lines denote multiplication by $V_1$, while blue dotted lines denote multiplication by $V_2$.}
\label{fig:complessi}
\end{figure}
\end{center}

\end{document}